\documentclass[11pt]{amsart}
\usepackage{mathtools} \mathtoolsset{showonlyrefs}
\usepackage{enumitem}
\usepackage{color}
\usepackage{listings,mathrsfs,comment}

\newtheorem{theorem}{Theorem}
\newtheorem{lemma}[theorem]{Lemma}
\newtheorem{corollary}[theorem]{Corollary}
\newtheorem{definition}[theorem]{Definition}
\newtheorem{remark}[theorem]{Remark}
\newtheorem{proposition}[theorem]{Proposition}
\def \R{\mathbb R}
\def \N{\mathbb N}
\newcommand {\B}[1][n] {B_2^{#1}}
\def \co{\operatorname{conv}}
\def \s{\mathbb{S}^{n-1}}

\renewcommand {\S}{\bar S_v}
\newcommand {\So}{\bar S_v^{\perp}}
\renewcommand{\epsilon}{\varepsilon}

\renewcommand{\chi}{\operatorname{1}}
\newcommand{\vol}[2][n]{\left|#2 \right|_{#1}}
\newcommand{\Vol}[1]{\vol[nm]{#1}}

 \usepackage[bookmarks=true,
bookmarksnumbered=true, breaklinks=true,
pdfstartview=FitH, hyperfigures=false,
plainpages=false, naturalnames=true,
colorlinks=false,
pdfpagelabels]{hyperref}

\title[On the polar of Schneider's difference body]{On the polar of Schneider's difference body}

\author[HADDAD]{Juli\'an Haddad}
\address{Departamento de An\'alisis Matem\'atico, Universidad de Sevilla, 
Sevilla, Spain 29208}
\email{jhaddad@us.es}

\author[LANGHARST]{Dylan Langharst}
\address{Carnegie Mellon University, Department of Mathematical Sciences, Pittsburgh, PA 15213, USA}
\email{dlanghar@andrew.cmu.edu}

\author[LIVSHYTS]{\\ Galyna V. Livshyts}
\address{Georgia Institute of Technology, School of Mathematics, Atlanta, GA, 30332 USA}
\email{glivshyts6@math.gatech.edu}

\author[PUTTERMAN]{Eli Putterman}
\address{Institut de Mathématiques de Jussieu-Paris Rive Gauche, Sorbonne Université, 4 place Jussieu, 75005 Paris, France}
\email{putterman@imj-prg.fr}

\begin{document}
\begin{abstract}
    In 1970, Schneider introduced the $m$th-order extension of the difference body $DK$ of a convex body $K\subset\R^n$, the convex body $D^m(K)$ in $\R^{nm}$. He conjectured that its volume is minimized for ellipsoids when the volume of $K$ is fixed. 
    In this work, we solve a dual version of this problem: we show that the volume of the polar body of $D^m(K)$ is maximized precisely by ellipsoids. For $m=1$ this recovers the symmetric case of the celebrated Blaschke-Santal\'o inequality. We also show that Schneider's conjecture cannot be tackled using standard symmetrization techniques, contrary to this new inequality. As an application of our results, we prove Schneider's conjecture asymptotically \'a la Bourgain-Milman. We also consider a functional version.
    \thanks{MSC 2020 Classification: 52A40;  Secondary: 28A75.
Keywords: Schneider's conjecture, Blaschke-Santal\'o inequality, polarity}
\end{abstract}

\maketitle
\section{Introduction}
Given a convex body $K$ (a compact, convex set with non-empty interior) in $\R^n$, one has the re-knowned Rogers-Shephard inequality \cite{RS57}:
\begin{equation}
\label{eq:RS}
\vol{DK}\vol{K}^{-1} \leq \binom{2n}{n},\end{equation}
with equality if and only if $K$ is an $n$-dimensional simplex. Here, $$DK=\{x\in\R^n: K\cap(K+x)\neq \emptyset\}=K+(-K)$$ is the difference body of $K$, $A+B=\{a+b:a\in A,b\in B\}$ is the Minkowski sum of sets $A$ and $B$ of $\R^n$, and $\vol{\cdot}$ denotes the Lebesgue measure on $\R^n$.

For $m\geq 1$, R. Schneider introduced in \cite{Sch70} an $m$th-order analogue of this inequality. Firstly, he defined the $m$th-order difference body $D^m(K)\subset \R^{nm}$ as 
\[D^m(K) = \left\{(x_1,\dots,x_m)\in(\R^n)^m: K\cap \bigcap_{i=1}^m (K+x_i)\neq \emptyset\right\}.\]
Then, he showed the following generalization of the Rogers-Shephard inequality:
\begin{equation}S_{n,m}(K):=\vol[nm]{D^m(K)}\vol[n]{K}^{-m} \leq \binom{nm+n}{n},
\label{eq:mRS}
\end{equation}
again with equality if and only if $K$ is an $n$-dimensional simplex.
Recently, M. Roysdon \cite{Ro20} established a generalization of \eqref{eq:mRS} involving measures that have radially decreasing density.

We recall that $A\subset \R^n$ is said to be origin-symmetric if $A=-A$, and merely symmetric if a translate of $A$ is origin-symmetric. It is natural to ask: what is the sharp lower bound of $S_{n,m}(K)$ among all convex bodies? If $m=1$ and $n\in\N$, or $n=2$ and $m\in\N$, then the lower-bound is obtained by all symmetric convex bodies. For the range $n\geq 3$ and $m\geq 2$, this is not the case. Schneider's conjecture states that $S_{n,m}(K) \geq S_{n,m}(\B),$ with equality if and only if $K$ is an ellipsoid for $n$ and $m$ in this range. 

A common method of proving isoperimetric-type inequalities, with an extremizer being a ball, is by using the technique of Steiner symmetrization. In this paper, we study how one may apply Steiner symmetrization in a few of its guises to tackle Schneider's conjecture. The takeaway is that most of the well-known variants of Steiner symmetrization, including a more recent one called fiber symmetrization, cannot be used to prove Schneider's conjecture. This is the main content of Section~\ref{sec:counter}. However, we are able to use these techniques to solve a problem that is dual to Schneider's conjecture. 

Our main result is the following inequality, which we call the Polar Schneider inequality, as it is the dual counterpart to Schneider's conjectured inequality. We recall that the polar of a convex body $K\subset\R^n$ containing the origin is given by
\[K^\circ=\{x\in\R^n:\langle x,y \rangle \leq 1, \, \forall y\in K\}.\]
We introduce the notation: $D^{m,\circ}(K)=(D^m(K))^\circ$.

\begin{theorem}
	\label{t:res_schneider_polar}
	Fix $m,n\in\N$. Let $K \subseteq \R^n$ be a convex body. Then, 
    \[\Vol{D^{m,\circ} (K)}\vol{K}^m \leq \Vol{D^{m,\circ} (\B)}\vol{\B}^m,\]
    with equality if and only if $K$ is an ellipsoid.
\end{theorem}
Here, $\B$ denotes the centered unit Euclidean ball; note that $D^{m}(\B)$ is not a dilate of $B_2^{nm}$ when $m>1$. Since $DK=2K$ for an origin-symmetric convex body, the $m=1$ case of Theorem~\ref{t:res_schneider_polar} for such bodies is the origin-symmetric case of the famous Blaschke-Santal\'o inequality (see e.g. the survey \cite{FMZ23}): for any convex body $K$ in $\R^n$ such that $K$ or $K^\circ$ has center of mass at the origin, it holds 
\begin{equation}
\label{eq:BS}
\vol{K}\vol{K^\circ} \leq \vol{\B}^2,
\end{equation}
with equality if and only if $K$ is a centered ellipsoid. We remark that there exists a unique $s(K) \in \R^n$, known as the Santal\'o point of $K$, so that $(K-s(K))^\circ$ has center of mass at the origin.
Conversely, one may note that
\begin{equation}
\vol{(DK)^\circ}\leq 2^{-n}\vol{K^\circ},
\label{eq:reverse_BM}
\end{equation}
which follows from writing 
\[(D K)^\circ = ((K^\circ)^\circ + (-K^\circ)^\circ)^\circ\]
and using the dual-Brunn-Minkowski inequality (cf \cite[Eq (85)]{G20}). Combining \eqref{eq:BS} and \eqref{eq:reverse_BM} then yields the $m=1$ case of Theorem~\ref{t:res_schneider_polar}. 

We prove Theorem~\ref{t:res_schneider_polar} using the Rogers-Brascamp-Lieb-Luttinger inequality in Section~\ref{sec:RBLL}. We will supply two additional proofs of Theorem~\ref{t:res_schneider_polar} in the case of origin-symmetric $K$ that are of independent interest. For example, in the second proof we will use shadow systems, whose definition we save for Section~\ref{sec:og_shadows}. 

For the third proof, which is the content of Section~\ref{sec:fiber}, we use an analogue of Steiner symmetrization called fiber symmetrization. This symmetrization has been used recently in a series of works exploring Schneider's $m$th-order theory \cite{HLPRY25,JH23,HLPRY23_2,YZZ25}. It is part of a larger symmetrization process introduced by P. McMullen \cite{McM99}, Bianchi-Gardner-Gronchi \cite{BGG17} and J. Ulivelli \cite{UJ23}.

As a direct corollary of our results, we obtain the following Bourgain-Milman-type inequalities for Schneider's conjecture in Section~\ref{sec:BM}.  To make sense of the result, note that, for $m \geq 2$, $D^m(K)$ is origin-symmetric if and only if $K$ is symmetric: this follows from \eqref{eq:m_diff} below.
\begin{corollary}
\label{t:schneider_polar_BM}
    Fix $n \geq 3,m\geq 2$. Let $K$ be a convex body in $\R^{n}$. Then, the inequality
    $$S_{n,m}(K) \geq c^{nm} \left(\pi nm\right)S_{n,m}(\B)$$
    holds with $c=\frac{1}{4}$ in general and with $c=\frac{1}{2}$ when $K$ is symmetric.
\end{corollary}

Finally, in Section~\ref{sec:functions}, we establish a functional version of Theorem~\ref{t:res_schneider_polar}. For a non-identically zero function $f:\R^n\to [0,\infty)$, its polar is given by
\begin{equation}
    \label{eq:polar_fun}
    f^\circ(x)=\inf_{y\in \R^n}\frac{e^{-\langle x,y\rangle}}{f(y)}.
\end{equation} 
For a convex body $K$ containing the origin, one has that \begin{equation}\label{eq:polar_of_bodies}[e^{-\|\cdot\|_K^2/2}]^\circ (x) = e^{-\|x\|^2_{K^\circ}/2}, \quad \text{and} \quad \chi_K^\circ(x) = e^{-\|x\|_{K^\circ}}.\end{equation} Here, $\|x\|_K=\inf\{t>0:x\in tK\}$ is the \textit{gauge function} of $K$, and one has $h_K=\|\cdot\|_{K^\circ}$. Also, $\chi_K$ is the usual characteristic function of $K$: $\chi_K(x) = 1$ if $x\in K$ and zero otherwise.

K. Ball \cite{Ball_thesis} extended the Blaschke-Santal\'o inequality to even, integrable, log-concave functions. We recall that a  function $f:\R^n\mapsto\R_+$ is $\log$-concave if for every $x,y$ such that $f(x)f(y)>0$, it holds
$$f((1-\lambda)x+\lambda y) \geq f(x)^{1-\lambda}f(y)^\lambda.$$
Regardless of the choice of $f$, one has that $f^\circ$ is $\log$-concave. Since $f^{\circ\circ\circ}=f^{\circ}$ and $f \leq f^{\circ\circ}$, with equality when $f$ is log-concave and upper-semi-continuous, K. Ball's extension handles all even, integrable functions.

In the seminal work \cite{AAKM04}, S. Artstein, B. Klartag and V. Milman formally introduced the definition of polarity of a function and extended K. Ball's result to integrable functions; a major hurdle was handling how to shift a function so that the integral of its polar function is finite. For a non-identically zero function $f$ on $\R^n$, if either $f$ or $f^\circ$ has center of mass at the origin, then
\begin{equation}
\label{eq:fun_sant}
\int_{\R^n}f(x)dx\int_{\R^n}f^\circ(x)dx \leq (2\pi)^n,
\end{equation}
with equality if and only if $f$ is Gaussian. To elucidate how one can modify $f$ so that $f^\circ$ has center of mass at the origin, it was shown by Lehec \cite{JL09} that there exists a unique $s(f)\in\R^n$, the so-called Santal\'o point of $f$, so that $(\tau_{s(f)} f)^{\circ}$ has center of mass at the origin, where $\tau_z f(x)=f(x-z)$.

In the same spirit, we establish a functional version of Theorem~\ref{t:res_schneider_polar} in Section~\ref{sec:functions}.
\begin{theorem}
\label{t:functional_singular_polar_sch}
     Let $f:\R^n\to\R_+$ be an integrable function such that either $f$ or $f^\circ$ has center of mass at the origin. Then,
     \[
     \begin{split}
     \left(\int_{\R^n}f^\frac{m+1}{m}\right)^m &\cdot \int_{\R^{nm}}\left(\prod_{i=1}^m f^\circ(x_i)\right)\left(f^\circ\left(-\sum_{i=1}^m x_i\right)\right)dx
     \\
     &\leq \left(\frac{2\pi m}{m+1}\right)^\frac{mn}{2}\left(\frac{2\pi}{(m+1)^\frac{1}{m}}\right)^\frac{nm}{2},
     \end{split}
     \]
     with equality if and only if there exists an $n\times n$ positive-definite symmetric matrix $A$ and a constant $c>0$ such that $f(x)=ce^{-\frac{1}{2}\langle Ax,x\rangle}$.
\end{theorem}
When $m=1$, one obtains from Theorem~\ref{t:functional_singular_polar_sch} that
\begin{equation}
\label{eq:singular_m_1}
     \int_{\R^n}f(x)^2dx \int_{\R^{n}}f^\circ(x)f^\circ(-x)dx \leq \pi^n.
\end{equation}
Inequality \eqref{eq:singular_m_1} implies \eqref{eq:fun_sant} for log-concave, even functions. To see this, let $g$ be an even log-concave function and set in \eqref{eq:singular_m_1} $f=\sqrt{g}$. Notice that $\sqrt{g}$ and $(\sqrt{g})^\circ$ will also be even. It then follows from the definition \eqref{eq:polar_fun} that for such $g$ one has $(\sqrt{g})^\circ(x)^2 \geq g^\circ(2x)$; this estimate and a variable substitution then yields \eqref{eq:fun_sant} from \eqref{eq:singular_m_1}. 
We compare Theorem~\ref{t:res_schneider_polar} and Theorem~\ref{t:functional_singular_polar_sch} in Section~\ref{sec:compare}. We conclude this work by using Theorem~\ref{t:functional_singular_polar_sch} to establish a Poincar\'e-type inequality in Section~\ref{sec:poin}. 

\section{Steiner Symmetrization and Schneider's Conjecture}
\label{sec:counter}
We recall that $S_v K$ is the Steiner-symmetrization of a convex set $K$ in the direction of the unit vector $v$. We suppress the definition of $S_v K$, but it can be inferred by setting $m=1$ in \eqref{eq_sym_def} below. It is well known (see e.g. \cite[Theorem 6.6.5]{Web94}) that Steiner symmetrization preserves volume, i.e. $\vol{S_u K}= \vol{K}$, and there exists a sequence of directions $\{u_i\}$ such that, if $S_1 K=S_{u_1} K$ and $S_{i+1} K = S_{u_{i+1}}(S_{i} K),$ then $S_i K$ converges in the Hausdorff metric to $K^\ast$. Here, $K^\ast$ is the centered Euclidean ball with the same volume as $K$.

It is natural to ask if $\vol[nm]{D^m(K)}$ decreases as $K$ is more symmetric, i.e. if $K$ is replaced by a Steiner symmetrization of $K$. Perhaps surprisingly, this is not true. Let $Q_d:=[-1,1]^d$. This section is dedicated to proving the following concerning $D^m(Q_d)$. Let $\s$ denote the unit sphere in $\R^n$.
\begin{proposition}
\label{p:no_steiner}
   Let $v=\frac{1}{\sqrt{3}}(1,1,1)\in \mathbb{S}^{2}$ and let $C_3=S_v Q_3$ be the double cone generated by the Steiner symmetrization of $Q_3$ in this direction $v$. Then, $\vol[3m]{D^m(C_3)} > \vol[3m]{D^m(Q_3)}$ when $m\geq 2$.
\end{proposition}
To prove Proposition~\ref{p:no_steiner}, we will need to discuss a connection between Schneider's conjecture and a celebrated conjecture of Petty. The support function of a convex body $L$ is given by $h_L(u)=\sup_{y\in L}\langle u,y\rangle.$ If $L$ contains the origin in its interior, then $h_L$ is a (pseudo)-norm whose unit ball is $L^\circ$. 

For $u\in\s$, let $u^\perp=\{x\in\R^n:\langle x,u\rangle = 0\}$ denote the hyperplane through the origin orthogonal to $u$, and let $P_{u^\perp} K$ denote the orthogonal projection of a convex body $K$ onto $u^\perp$. Cauchy showed that there exists an origin-symmetric convex body, called the projection body $\Pi K$ of $K$, whose support function satisfies $h_{\Pi K}(u) = \vol[n-1]{P_{{u}^\perp} K}$. If $K$ is a planar convex body, then one has
$$4 \leq P_2(K):=\frac{\vol[2]{\Pi K}}{\vol[2]{K}} \leq 6,$$
with equality on the left-hand side when $K$ is symmetric, and equality on the right-hand side when $K$ is a triangle. In fact, this is merely a restatement of \eqref{eq:RS} when $n=2$: $\Pi K$ is just a rotation of $DK$ by $\pi/2$ for planar convex bodies \cite[Theorem 4.1.4, pg. 143]{gardner_book}.

Thus, like in the case of Schneider's conjecture for bounding $\vol[nm]{D^m(K)}$ from below, bounds on $\vol[n]{\Pi K}$ are only meaningful when $n\geq 3$. For $n\geq 3$, Petty's conjecture is precisely that the \textit{Petty product of $K$} \begin{equation}P_n(K):=\frac{\vol[n]{\Pi K}}{\vol[n]{K}^{n-1}}
\label{eq:petty_product}
\end{equation} is minimized by ellipsoids \cite{CMP71}. We remark that sharp lower-and-upper bounds for the volume of $\Pi^\circ K= (\Pi K)^\circ$, the polar projection body of $K$, are now classical; see \cite{Zhang91,petty61_1,CMP71}.

It will be shown in the forthcoming work by the last named author \cite{EP25} (it is also indirectly implied by \cite[Eq. 18]{Sch20}) that, for any symmetric convex body $K\subset \R^3$ and $m\geq 2$, one has the relation
\begin{equation}
4\cdot S_{3,m}(K)=84+3\cdot P_3(K).
\label{eq:eli}
\end{equation}
Similar expressions hold for $n\geq 3$, but there are additional terms in the summation. This formula connects Schneider's conjectured inequality for $D^m(K)$ and Petty's conjectured inequality for $\Pi K$. 

We can now prove Proposition~\ref{p:no_steiner}.
\begin{proof}[Proof of Proposition~\ref{p:no_steiner}]
    It was shown in \cite[Theorem 3]{CS11} that $\vol[3]{\Pi (C_3)} > \vol[3]{\Pi (Q_3)}$. Applying \eqref{eq:eli}, we obtain the claim.
\end{proof}

\begin{remark}
\label{re:no_work}
We can numerically verify Proposition~\ref{p:no_steiner} when $m=2$: using \eqref{eq:conditions}  below, one has that $$D^2 Q_1 = \{(x,y) \in \R^2: |x|,|y| \leq 1, |x-y| \leq 1 \},$$  which has volume $3$ by Proposition~\ref{p:Q_d} below, and so $\vol[6]{D^2(Q_3)}=27$. Observe also that $C_3$ has $8$ vertices, and the equation \eqref{eq:m_diff} yields that $D^2 (C_3)$ is a $6$-dimensional polytope whose $8^3 = 512$ vertices are of the form $(p+r,q+r)$ where $p,q,r$ are vertices of $C_3$. Using Polymake, one can then compute that $\vol[6]{D^2(C_3)}=27.75$.


\end{remark} 
Despite Proposition~\ref{p:no_steiner}, it was shown in \cite{LX24} that Steiner symmetrization can be used to solve other questions in Schneider's $m$th-order framework.

\begin{proposition}
\label{p:Q_d}
    Fix $m,n\in\N$. Then, $D^m(Q_n)\neq Q_{nm}$, but one does have $D^m (Q_n) = (D^m (Q_1))^n$.
\end{proposition}
\begin{proof}
The first claim can be easily seen by setting $n=1,m=2$. As for the second claim, it is clear that $(Q_n+x) \cap (Q_n+y) \neq \emptyset$ if and only if $|x^i - y^i| \leq 1$, where $x^i, y^i$ are the coordinates of $x$ and $y$ respectively.
A point $(x_1, \ldots, x_m) \in (\R^n)^m$ is in $D^m (Q_n)$ if and only if the following inequalities are satisfied
\begin{equation}
\label{eq:conditions}
|x_i^k - x_j^k| \leq 1, \quad |x_i^k| \leq 1, \quad \forall i,j,k \text{   with   } 1 \leq i,j\leq m, 1\leq k \leq n.\end{equation}
The claim follows.
\end{proof}

Another conjecture of R. Schneider is that, among symmetric convex bodies, the maximum of $P_n(K)$ from \eqref{eq:petty_product} is $2^n$, with equality if and only if $K$ is the affine image of Cartesian products of line segments or centrally symmetric planar convex figures. This was shown to be false by Brannen \cite{BNS96}. However, he conjectured it was true for the class of zonoids (i.e. limits in the Hausdorff metric of Minkowski sums of line segments). Saroglou verified this conjecture when $n=3$ (see \cite{CS11}). Thus, from \eqref{eq:eli}, the verification of the Brannen-Schneider conjecture by Saroglou implies the following.
\begin{theorem}
\label{t:eli}
    Let $Z$ be a zonoid in $\R^3$. Then, for all $m\in\N$, one has \begin{equation}
    \label{eq:eli_saroglou}
    S_{3,m}(Z) \leq S_{3,m}(Q_3).\end{equation}
    Equality holds if and only if $Z$ can be written as the Minkowski sum of five line
segments or as the sum of a cylinder and a line segment.
\end{theorem}


\section{The Polar Schneider inequality}
\label{sec:RBLL}
In the work \cite{Sch70}, R. Schneider considered a generalization of $D^m(K)$ to $(m+1)$ convex bodies. Let $\mathscr{K}=(K_0,K_1,\dots,K_m)$ be a collection of $(m+1)$ convex bodes in $\R^n$. Then, their $m$th-order difference body is the set
\begin{equation}
\label{eq:distinct_m+1_bodies}
D^m (\mathscr{K})= \left\{(x_1,\dots,x_m)\in(\R^n)^m: K_0\cap\bigcap_{i=1}^m (K_i+x_i)\neq \emptyset\right\}.\end{equation}
Then, he showed the Rogers-Shephard inequality for the collection $\mathscr{K}$:
\begin{equation}
\label{eq:distinct_higher_RS}
\frac{\vol[nm]{D^m(\mathscr{K})}\vol{\cap_{i=0}^mK_i}}{\prod_{i=0}^m\vol[n]{K_i}} \leq \binom{nm+n}{n}.
\end{equation}
We will work in this more general setting. We first need to determine the support function of $D^m(\mathscr{K})$. We define the diagonal embedding $\Delta_m:\R^n\to\R^{nm}$ by $\Delta_m(x)=(x,\dots,x)$.

\begin{proposition}
    Fix $m,n\in\N$. Let $K_0,K_1,\dots,K_m\subset\R^n$ be convex bodies, and define the collection $\mathscr{K}=(K_0,K_1,\dots,K_m)$. Then, the support function of $D^m(\mathscr{K})$ is given by
\begin{equation}
        h_{D^m(\mathscr{K})}((\theta_1,\cdots,\theta_m)) = h_{K_0}\left(\sum_{i=1}^m\theta_i\right)+ \sum_{i=1}^mh_{K_i}(-\theta_i)
        \label{eq:DM_sup}
    \end{equation}
\end{proposition}
\begin{proof}
    For a convex body $L\subset \R^n$, set for $i=1,\dots,m$, 
    \begin{equation}
    \label{eq:K_i}
    L\cdot e_i^t=\{o\}\times\cdots\times L\times\cdots\times\{o\} \subset \R^{nm}\end{equation}
    where $L$ is in the $i$th copy of $\R^n$ in the decomposition of $\R^{nm}$ as $m$ direct products of $\R^n$.
    We claim that
\begin{equation}
\label{eq:m_diff}
D^m(\mathscr{K})=\Delta_m(K_0) + \left(\sum_{i=1}^m \left(-K_i\right)e^t_i\right)=\Delta_m(K_0)+\prod_{i=1}^m(-K_i).
\end{equation}
    Since for any convex body $L\subset \R^n,$ 
    \[h_{\Delta_m L}((\theta_1,\cdots,\theta_m))=h_{L}\left(\sum_{i=1}^m\theta_i\right),\]
    formula \eqref{eq:DM_sup} immediately follows from \eqref{eq:m_diff}.

    As for \eqref{eq:m_diff}, an $m$-tuple of vectors $(x_1,\dots,x_m)\in D^m(\mathscr{K})$ if and only if there exists $z_0\in
    K_0\cap\bigcap_{i=1}^m (K_i+x_i)$.
    Therefore, there exists $z_i\in K_i$, $i=1,\dots,m$ such that $z_0=z_i+x_i$. Solving for $x_i$, we have $x_i=z_0-z_i.$ Thus, $$(x_1,\dots,x_m) \!=\!(z_0,\dots,z_0)+(-z_1,\dots,-z_m) \in \Delta_m(K_0)+\left(\sum_{i=1}^m \left(-K_i\right)e^t_i\right).$$ The converse is similar and the claim follows.
\end{proof}
\begin{remark}
\label{r:matrix_rep}
    The formula \eqref{eq:m_diff} for $D^m(\mathscr{K})$ shows that $D^m(\mathscr{K})\subset \R^{nm}$ is a shadow of $\prod_{i=0}^mK_i\subset \R^{n(m+1)}$. In fact, if we set $0_n$ as the $n\times n$ matrix of all zeros and define the $m\times (m+1)$ block matrix $$ P_m: =\begin{pmatrix} 
I_n &-I_n & 0_n &  \cdots & 0_n  \\
I_n & 0_n & -I_n &   \cdots & 0_n \\
\vdots & \vdots & \vdots & \ddots & \vdots \\
I_n &0_n & 0_n &  \cdots &-I_n 
\end{pmatrix}$$
where each block is $n\times n$, then $D^m(\mathscr{K}) = P_m\left(\prod_{i=0}^mK_i\right)$.
\end{remark}

As an application of Remark~\ref{r:matrix_rep}, we obtain a monotonicity of $S_{n,m}(K)$ in $m$. We recall the following consequence of the geometric Brascamp-Lieb inequality \cite[Equation 3.3]{GKZ23}: say we have subspaces $F_1,\dots, F_r$ of $\R^d$ and weights $c_1,\dots, c_r$ such that, if $P_{F_j}$ is the orthogonal projection onto $F_j$,
\[
sI_{d} = \sum_{j=1}^rc_jP_{F_j}
\]
for some $s>0$. Then, for every compact set $L\subset \R^d$, one has
\begin{equation}
\label{eq:geo_bl}
\vol[d]{L}^s \leq \prod_{j=1}^r\vol[\dim(F_j)]{P_{F_j}L}^{c_j}.\end{equation}
\begin{corollary}
    Let $m,n\in\N$ and let $K\subset\R^n$ be a convex body. Then, $$S_{n,m+1}(K)^m \leq S_{n,m}(K)^{m+1}.$$
\end{corollary}
\begin{proof}
    Define the (block) column matrix $0_{n,m}$ of length $m$ via $$0_{n,m} = \begin{pmatrix}
    0_n\\
    \vdots\\
    0_n
\end{pmatrix}.$$
With this definition, we have
\begin{equation}
\label{eq:nested}
P_{m+1}=\begin{pmatrix} 
\multicolumn{2}{c}{P_m} &   0_{n,m} \\
(I_{n} & 0_{n,m}^T )&  -I_n  
\end{pmatrix}.\end{equation}
If we write $\R^{nm+2n}=\R^{nm+n}\otimes \R^{n}$, then the second column acts on $\R^n$; the submatrix 
$\begin{pmatrix}
\multicolumn{2}{c}{P_m}\\
     I_n & 0_{n,m}^T \\
\end{pmatrix}$
is a non-singular $(m+1)\times (m+1)$ block matrix acting on $\R^{nm+n}$. Now, let $x\in D^{m+1}(K)$. By Remark~\ref{r:matrix_rep}, there exists $y\in K^{m+2}$ such that $x=P_{m+1}y$. We write $y=(y_0,\dots,y_{m},y_{m+1})$, where $y_i \in K$ for $i=0,\dots,m+1$. Then, by \eqref{eq:nested}, we have
\begin{align*}
x=P_{m+1}\begin{pmatrix}(y_0,\dots,y_{m})^T\\ y_{m+1}\end{pmatrix}=\begin{pmatrix}P_m (y_0,\dots,y_{m})^T \\ y_{0}-y_{m+1}\end{pmatrix}.
\end{align*}
We deduce that $P_{(\R^n_{m+1})^\perp}D^{m+1}(K) = D^m(K)$. Actually, this construction is invariant under permutations of the $\R^n$, and so we deduce $P_{(\R^n_j)^\perp}D^{m+1}(K) \simeq D^m(K)$, for all $j=1,\dots,m,m+1$. Next,
we apply \eqref{eq:geo_bl} with $d=nm+1$, $s=m$, $L=D^{m+1}(K)$, $F_j = (\R^n_j)^\perp$ and $c_j=1$, and obtain
\begin{equation}
    \vol[n(m+1)]{D^{m+1}(K)}^m \leq \vol[nm]{D^{m}(K)}^{m+1}.
\end{equation}
Dividing through by $\vol[n]{K}^{m(m+1)}$ yields the claim.
\end{proof}

\subsection{Proof of Theorem~\ref{t:res_schneider_polar} using the Rogers-Brascamp-Lieb-Luttinger inequality}
In the following theorem, we establish the Polar Schneider inequality for a collection $\mathscr{K}$. 
\begin{theorem}
\label{t:res_schneider}
    Fix $m,n\in \N$. Let $K_0,K_1,\dots,K_m\subset \R^n$ be $(m+1)$ convex bodies with center of mass at the origin, under the constraint that $\vol{K_i^\circ}=\vol{K_0^\circ}$ for all $i$. Define the collection $\mathscr{K}=(K_0,K_1,\dots,K_m)$. Then,
    \begin{equation*}
        \vol[nm]{D^{m,\circ}(\mathscr{K})} \prod_{i=1}^m\vol{K_i}\leq \vol[n]{\B}^m  \vol[nm]{D^{m,\circ}(\B)}.
    \end{equation*}
    A necessary condition for equality is that each $K_i$ is a centered ellipsoid, and a sufficient condition is when each $K_i$ is the same centered ellipsoid.
\end{theorem}
As one can see, there is a slight gap between the necessary and sufficient conditions for equality. This is because of the complicated nature of the equality conditions for the Rogers-Brascamp-Lieb-Luttinger inequality; we do not know if there is equality in Theorem~\ref{t:res_schneider} when all the $K_i$ are centered ellipsoids, but at least one of them is different than the others.

As an operator on convex bodies, $D^m$ is $1$-homogeneous, i.e. for every $t\in\R$, one has that $D^m(tK)=t D^m(K)$. However, as an operator on collections $\mathscr{K}$ of $(m+1)$ convex bodies, the homogeneity only holds if we dilate each $K_i\in \mathscr{K}$ by the same factor $t$. Therefore, the requirement $\vol{K_i^\circ}=\vol{K_0^\circ}$ in Theorem~\ref{t:res_schneider} should be viewed as a homogeneity constraint. We now show how Theorem~\ref{t:res_schneider} implies Theorem~\ref{t:res_schneider_polar}.

\begin{proof}[Proof of Theorem~\ref{t:res_schneider_polar}]
First observe that $D^m(K+x)=D^m(K)$ for all $x\in\R^n$. We may therefore assume that $K$ has center of mass at the origin. Then, the claim of the inequality is immediate from Theorem~\ref{t:res_schneider} by setting $K_i=K$ for all $i=0,1\dots,m$. For the equality characterization, we simply note that when all the $K_i$ are the same, there is equality in Theorem~\ref{t:res_schneider} if and only if $K$ is an ellipsoid.
\end{proof}

Recall that a nonnegative, measurable function $f$ on $\R^n$ can be written as, for $x\in\R^n$,
\[
f(x)=\int_0^\infty \chi_{\{f\geq t\}}(x)dt,
\]
where $\{f\geq t\}=\{x\in\R^n:f(x)\geq t\}$ are the superlevel sets of $f$. If $f$ has the additional property that almost all of its superlevel sets have finite volume, then the symmetric decreasing rearrangement of $f$ is the function $f^\ast$ given by
$$f^\ast(x)=\int_0^\infty \chi_{\{f\geq t\}^\ast}(x)dt.$$

The following inequality was shown independently by Rogers \cite{RCA57} and Brascamp, Lieb and Luttinger \cite{BLL74}. 
\begin{proposition}
\label{p:RBLL}
    Let $k,m\geq 1$ and $f_i:\R^n\to\R$ with $1\leq i \leq k,$ be nonnegative and measurable functions, and let $a^{(i)}_j$ be real numbers with $1\leq j \leq m, 1\leq i \leq k$. Then, for $v\in\s,$
    \begin{align*}
    \int_{\R^n}\cdots\int_{\R^n}\prod_{i=1}^k f_i&\left(\sum_{j=1}^m a_j^{(i)}x_j\right)dx_1\dots dx_m
    \\
    &\leq \int_{\R^n}\cdots\int_{\R^n}\prod_{i=1}^kf_i^\ast\left(\sum_{j=1}^m a_j^{(i)}x_j\right)dx_1\dots dx_m.
    \end{align*}
\end{proposition}
We will need the following rudimentary fact. Recall that a measure $\mu$ has density if there exists a locally integrable, nonnegative function $\varphi$ such that $\frac{d\mu(x)}{dx}=\varphi(x)$. We say a nonnegative function $\varphi$ is $q$-homogeneous, $q\in \R$, if $\varphi(tx)=t^q\varphi(x)$ for $t>0$ and $x\in\R^n\setminus\{o\}$. 
\begin{proposition}
\label{p:homo_form}
    Let $L\subset\R^d$ be a convex body containing the origin in its interior. Let $\mu$ be a Borel measure on $\R^d$ with $q$-homogeneous density. Then, for every $p> 0,$ one has
    \begin{equation}
        \mu(L) = \frac{1}{p^\frac{q+d}{p}\Gamma\left(1+\frac{q+d}{p}\right)}\int_{\R^d} e^{-\frac{\|x\|_L^p}{p}}d\mu(x).
        \label{eq:measure_p_homo}
    \end{equation}
    In particular, applying this to $L^\circ$, one has
    \begin{equation}
        \mu(L^\circ) = \frac{1}{p^\frac{q+d}{p}\Gamma\left(1+\frac{q+d}{p}\right)}\int_{\R^d} e^{-\frac{h_L(x)^p}{p}}d\mu(x).
        \label{eq:measure_p_homo_circ}
    \end{equation}
\end{proposition}
\begin{proof}
    Observe that, since $\mu$ has $q$-homogeneous density $\varphi$, then $\mu$ is $(q+d)$-homogeneous. Indeed, for a Borel set $A\subset \R^d$ and $t>0$, we have
    \begin{align*}
        \mu(tA)=\int_{tA}\varphi(x)dx = t^d\int_{A}\varphi(tx)dx 
        = t^{d+q}\int_A\varphi(x)dx = t^{d+q}\mu(A).
    \end{align*}
    Then, one has from Fubini's theorem
    \begin{align*}
        \int_{\R^d} e^{-\frac{\|x\|_L^p}{p}}d\mu(x) &= \int_{\R^d}\int_{\frac{\|x\|_L^p}{p}}^\infty e^{-t}d\mu(x) 
        \\
        &= \int_0^\infty \mu\left(\left\{x\in\R^d:\frac{\|x\|_L^p}{p} \leq t\right\}\right)e^{-t}dt
        \\
        &= \int_0^\infty \mu\left((pt)^\frac{1}{p}L\right)e^{-t}dt 
        = \mu(L)p^\frac{q+d}{p}\int_0^\infty e^{-t}t^{\frac{q+d}{p}}dt
        \\
        & = \mu(L)p^\frac{q+d}{p}\Gamma\left(1+\frac{q+d}{p}\right).
    \end{align*}
    This establishes \eqref{eq:measure_p_homo}. The equation \eqref{eq:measure_p_homo_circ} follows by replacing $L$ with $L^\circ$ and using that $\|\cdot\|_{L^\circ}=h_L$.
\end{proof}

We now prove a further generalization of Theorem~\ref{t:res_schneider}, which has the weight $\prod_{i=1}^m|x_i|^{-q}$ on $D^{m,\circ}(\mathscr{K}).$ We restrict to $q\in [0,n)$ for integrability, and note that the function $f(x)= |x|^{-q}$ is invariant under decreasing rearrangements, i.e. $f^\ast=f$. This is because it is rotational invariant and quasi-concave, where the latter means its superlevel sets are convex.

\begin{lemma}
\label{l:res_schneider}
    Fix $m,n\in \N$. Let $K_0,K_1,\dots,K_m\subset \R^n$ be $(m+1)$ be convex bodies with center of mass at the origin, under the constraint that $\vol{K_i^\circ}=\vol{K_0^\circ}$ for all $i$. Define the collection $\mathscr{K}=(K_0,K_1,\dots,K_m)$. Fix $q\in [0,n)$ and let $\mu$ be the Borel measure on $\R^{nm}$ with density $(x_1,\dots,x_m)\mapsto\prod_{i=1}^m|x_i|^{-q}$. Then, one has
    \begin{equation}
    \label{eq:ineq_volume_best}
        \mu\left(D^{m,\circ}(\mathscr{K})\right) \prod_{i=1}^m\vol{K_i}^{1-\frac{q}{n}}\leq \vol[n]{\B}^{m(1-\frac{q}{n})}  \mu\left(D^{m,\circ}(\B)\right).
    \end{equation}
    A necessary condition for equality is that each $K_i$ is a centered ellipsoid. When $\mu$ is the Lebesgue measure, i.e. $q=0$, a sufficient condition for equality is when all $K_i$ are the same centered ellipsoid.
\end{lemma}

\begin{proof}
Notice that the density of $\mu$ is $(-qm)$-homogeneous. From \eqref{eq:measure_p_homo_circ} with $d=nm,$ $L=D^m(\mathscr{K}),$ $p=1$ and $q$ replaced by $-qm$, one obtains that
\begin{equation}
\begin{split}
    &\mu\left(D^{m,\circ}(\mathscr{K})\right)=\frac{1}{\Gamma\left(1+m(n-q)\right)}\int_{\R^{nm}}e^{-h_{ D^{m}(\mathscr{K})}(x)}d\mu(x)
    \\
    &=\frac{1}{\Gamma\left(1+m(n-q)\right)} \int_{\R^{n}}\cdots\int_{\R^n} e^{-h_{K_0}\left(\sum_{i=1}^mx_i\right)}  \prod_{i=1}^m e^{-h_{-K_i}(x_i)}|x_i|^{-q}dx_i,
\end{split}
    \label{eq:gauge_volume}
\end{equation}
where we used \eqref{eq:DM_sup} for the formula of the support function of $D^{m}(\mathscr{K})$.

Let $B$ be the centered Euclidean ball such that the volume of $\{h_B \leq 1\}= B^\circ$ is the same as $\{h_{K_i} \leq 1\}=K_i^\circ$ for all $i=0,\dots,m$. If $B=r\B$, then $r$ is defined via $\vol[n]{K_i^\circ} = r^{-n}\vol[n]{\B}$ for all $i=0,\dots,m$. Then, we have from \eqref{eq:gauge_volume} that iterating Proposition~\ref{p:RBLL} (with $k=2m+1$) yields
\begin{align}
    \mu\left(D^{m,\circ}(\mathscr{K})\right) \leq \mu\left(D^{m,\circ}(B)\right) = r^{-m(n-q)} \mu(D^{m,\circ}(\B)),
\end{align}
    where we used that $D^{m,\circ}$ is $(-1)$-homogeneous, e.g. $$D^{m,\circ}(r\B)=r^{-1}D^{m,\circ}(\B),$$ and $\mu$ is $(nm-qm)$-homogeneous. From the definition of $r$, we obtain
    \begin{equation}
        \mu\left(D^{m,\circ}(\mathscr{K})\right)\!\! \prod_{i=1}^m\vol{K_i}^{1-\frac{q}{n}}\!\leq\! \vol[n]{\B}^{m(1-\frac{q}{n})}  \!\mu\left(D^{m,\circ}(\B)\right)\!\!\prod_{i=1}^m\!\!\left(\!\!\frac{\vol[n]{K_i}\vol[n]{K_i^\circ}}{\vol[n]{\B}^2}\right)^{1-\frac{q}{n}}\!\!\!\!\!\!\!\!\!.
        \label{eq:BS_with_PS}
    \end{equation}
    We remark that the proof until this point does not require the $K_i$ to have center of mass at the origin, only that they contain the origin in their interiors. By applying the Blaschke-Santal\'o inequality \eqref{eq:BS} to \eqref{eq:BS_with_PS}, we obtain the claimed inequality. As for the equality case, note that for any linear transformation $A$,
    \[
     \vol[nm]{D^{m,\circ}(A\mathscr{K})} \prod_{i=1}^m\vol{AK_i}=\vol[nm]{D^{m,\circ}(\mathscr{K})} \prod_{i=1}^m\vol{K_i},
    \]
    where $A\mathscr{K} = (AK_0,\dots,AK_m).$
\end{proof}

    The Theorem~\ref{t:res_schneider} then follows from Lemma~\ref{l:res_schneider} by taking $\mu$ to be the Lebesgue measure. When $K$ is origin-symmetric and $m=1$, Lemma~\ref{l:res_schneider} recovers \cite[Proposition 6.5]{CKLR24}.

    \subsection{A connection to the Dual Brunn-Minkowski Theory}
    \label{sec:dual}
    Throughout this subsection, fix $m,n\in\N$ and $q\in [0,n)$ and a convex body $K\subset \R^n$ containing the origin. In \cite{Lut75}, E. Lutwak introduced the $q$th dual Quermassintegral, which were expanded upon by R. Gardner \cite{RJG07}:
    \begin{equation}
        \widetilde W_{q}(K) := \frac{1}{n}\int_{\s} \rho_K(u)^{n-q}du=\frac{n-q}{n}\int_{K}|x|^{-q}dx,
        \label{eq:dual_quer}
    \end{equation}
    where the second equality is polar coordinates. Manifestly, one has the relation $\widetilde W_{0}(K)=\vol[n]{K}$. By an application of H\"older's inequality, one has the dual isoperimetric inequalty:
\begin{equation}
\vol[n]{K}^{1-\frac{q}{n}} \geq \widetilde W_q(K)\vol[n]{\B}^{-\frac{q}{n}},
\label{eq:gardner_ineq}
\end{equation}
with equality if and only if $q=0$ or $q\neq 0$ and $K$ is a dilate of $B_2^n$. Combining \eqref{eq:gardner_ineq} with \eqref{eq:ineq_volume_best}, we obtain, under the same assumptions as Lemma~\ref{l:res_schneider}, the inequality
\begin{equation}
    \label{eq:ineq_mu_best}
        \mu\left(D^{m,\circ}(\mathscr{K})\right) \prod_{i=1}^m\widetilde W_q(K_i)\leq \widetilde W_q(\B)^m  \mu\left(D^{m,\circ}(\B)\right).
    \end{equation}
    The equality statement is the same if $q=0$; for $q\in (0,n)$ change every instance of ``ellipsoid'' to ``Euclidean ball''. We explicitly list as a corollary of \eqref{eq:ineq_mu_best} the case when $m=1$ and $K$ is origin-symmetric; in which case, $\mu$ and $\widetilde{W}_q$ coincide up to the constant $\frac{n-q}{n}$. We therefore recover the Blaschke-Santal\'o inequality for dual Quermassintegrals from \cite[Corollay 4.1]{FLM23}.
    \begin{corollary}
        Let $K\subset\R^n$ be an origin-symmetric convex body. Then, for $q\in [0,n)$,
        \[
        \widetilde{W}_q(K)\widetilde{W}_q(K^\circ) \leq \widetilde{W}_q(\B)^2,
        \]
        with equality if and only if $K$ is a centered ellipsoid ($q=0$) or a centered Euclidean ball ($q\in (0,n)$).
    \end{corollary}

\subsection{A Bourgain-Milman inequality for Schneider's conjecture}
\label{sec:BM}
We can now prove Corollary~\ref{t:schneider_polar_BM}. We recall that \begin{equation}
   \vol[n]{\B}=\frac{\pi^\frac{n}{2}}{\Gamma\left(1+\frac{n}{2}\right)},
   \label{eq:ball_volume}
\end{equation}
where $\Gamma(\cdot)$ is the usual Gamma function. We also need the improved version of J. Bourgain's and V. Milman's inequality \cite{BM87} by G. Kuperberg {\cite{GK08}}:
    Let $L$ be a convex body in $\R^d$. Then, the inequality
    \begin{equation}
    \left(\frac{c_1}{d}\right)^d \leq \vol[d]{L}\vol[d]{L^\circ}
    \label{eq:VM}
\end{equation}
holds with $c_1=\frac{\pi e}{2}$; the larger bound $c_1=\pi e$ holds when $L$ is origin-symmetric.

 Usually, \eqref{eq:VM} requires that $L$ must be translated so that $L^\circ$ has center of mass at the origin. Such $L$ are said to be in \textit{Santal\'o position}. They have the additional property that $\vol[d]{L^\circ}\leq \vol[d]{(L-z)^\circ}$ for all $z\in\R^d$. Thus, \eqref{eq:VM} holds regardless of which translate of $L$ is taken. This is vital for us, as we do not know if $D^m(K)$ or $D^{m,\circ}(K)$ has center of mass at the origin. Since $D^m(K)=D^m(K+x)$ for all $x\in\R^n$, the center of mass of $D^m(K)$ is somehow independent of the center of mass of $K$.

 \begin{proof}[Proof of Corollary~\ref{t:schneider_polar_BM}]
We have from Theorem~\ref{t:res_schneider_polar} and the Blaschke-Santal\'o inequality \eqref{eq:BS} applied to $D^m(K)$,
\begin{equation}
\label{eq:BM_for_schneider}
\begin{split}
\frac{\vol[n]{K}^{-m}\vol[nm]{D^{m}(K)}}{\vol[n]{\B}^{-m}  \vol[nm]{D^{m}(\B)}} 
& \geq  \frac{\vol[nm]{D^{m}(K)}\vol[nm]{D^{m,\circ}(K)}}{\vol[nm]{D^{m}(\B)}\vol[nm]{D^{m,\circ}(\B)}}
\\
& \geq  \vol[nm]{B_2^{nm}}^{-2}{\vol[nm]{D^{m}(K)}\vol[nm]{D^{m,\circ}(K)}}.
\end{split}
\end{equation}
Then, use \eqref{eq:VM} in $\R^{nm}$ for the body $D^{m}(K)$ and \eqref{eq:ball_volume} to obtain
$$\frac{\Vol{D^m(K)}\vol{K}^{-m}}{\Vol{D^m(\B)}\vol{\B}^{-m}} \geq \left(\left(\frac{c_1}{nm\pi}\right)^{nm}\Gamma\left(1+\frac{nm}{2}\right)^2\right).$$

Recall that a version of Stirling's formula states for $x>0$ one has
\begin{equation}
\label{eq:stirling}
\sqrt{2\pi} \, x^{x + 1/2} e^{-x} \leq \Gamma(1+x) \leq \sqrt{2\pi} \, x^{x + 1/2} e^{-x} e^{\frac{1}{12x}}.
\end{equation}
Then, the claim follows from \eqref{eq:stirling} with $x=nm/2$.
\end{proof}
A conjecture by K. Mahler \cite{MK39} implies that \eqref{eq:VM} should hold with the constant $\tilde{c_1}=e^2$ in general and $\tilde{c_1}=4e$ in the origin-symmetric case. In which case, Corollary~\ref{t:schneider_polar_BM} would improve to the constant $\tilde c=\frac{1}{2}\frac{e}{\pi}$ in general and $\tilde c=\frac{2}{\pi}$ in the origin-symmetric case. In fact, the verification of Mahler's conjecture for zonoids by S. Reisner (see \cite{Reis85,Resi86}) shows that we can use $\tilde c$ in Corollary~\ref{t:schneider_polar_BM} when $K$ is a zonoid.

It is natural to ask concerning a lower bound for $\vol[nm]{D^{m,\circ}(K)} \vol{K}^m$; a non-sharp bound immediately follows from \eqref{eq:VM} and \eqref{eq:mRS}. In the case of zonoids in $\R^3$, one can alternatively use Theorem~\ref{t:eli} instead of \eqref{eq:mRS} to get a sharper, but still strict (due to Proposition~\ref{p:Q_d}) lower-bound.

\section{Additional proofs of Theorem~\ref{t:res_schneider_polar} under symmetry assumptions}
\subsection{Proof of the origin-symmetric case using shadow systems}
\label{sec:og_shadows}
In this section, we apply shadow systems to the object $D^{m,\circ}(\mathscr{K})$ defined via \eqref{eq:distinct_m+1_bodies}. Shadow systems were first introduced by Rogers and Shephard \cite{RS58:2,She64}: given a convex body $K\subset\R^n$, a shadow system of $K$ in the direction $v\in\s$ with (bounded) speed function $\alpha:K\mapsto \R$ is a family of convex sets $K(t)\subset \R^n$ given by
\begin{equation}
    \label{def_shadowsystem}
    K(t) = \co \left\{x+\alpha(x)tv:x\in K, t\in [-a,b], a,b>0\right\}.
\end{equation}
Here, $\co$ denotes the convex hull operation. Our notation $K(t)$ suppresses the direction $v$, the function $\alpha$ and the interval $[-a,b]$. If one defines
    $$\tilde K = \text{conv}\{x+\alpha(x)e_{n+1}:x\in K\}\subset \R^{n+1},$$
then 
\begin{equation}
h_{K(t)}(u)=h_{\tilde K}(u+t\langle u,v \rangle e_{n+1}).
\label{eq:lifted_relation}
\end{equation} 
They further established that $\vol{K(t)}$ is a convex function in the variable $t$. Campi and Gronchi \cite{CG06} later established that, if $K$ is origin-symmetric, then $\vol{K(t)^\circ}^{-1}$ is also convex in the variable $t$. Combining this fact with the following observation by Shephard \cite{She64} yields another proof of the Blasckhe-Santal\'o inequality \eqref{eq:BS} in the origin-symmetric case.
\begin{proposition}
\label{p:shep}
    Fix a collection $\mathcal{K}$ of convex bodies in $\R^n$ under the following condition: $\mathcal{K}$ contains all dilates of $B_2^n$ and, for every $K\in\mathcal{K}$, every shadow system $K(t)$ of $K$ satisfies $K(t)\in\mathcal{K}$.
    Next, let $F:\mathcal{K}\mapsto \R_+$ be a functional that is continuous in the Hausdorff metric and reflection invariant such that $F(K(t))$ is convex in the parameter $t$. Then, for a fixed $C\in\R_+$ the solution to
    $\min\{F(K): K\in\mathcal{K},\vol{K}=C\}$
    is obtained at the Euclidean ball whose volume is $C$.
\end{proposition}
 We pause to remark that Meyer and Reisner \cite{MR06} extended Campi's and Gronchi's result to the case when $K^\circ$ has center of mass at the origin, but their approach is difficult to extend to our setting.

We provide another proof of the origin-symmetric case of Theorem~\ref{t:res_schneider_polar} by establishing the following analogous fact. 
\begin{theorem}
\label{t:convex_t}
    Fix $m,n\in\N$. Let $K_0,\dots,K_m\subset\R^n$ be $(m+1)$ origin-symmetric convex bodies. Let, for $i=0,1,\dots,m,$ $K_i(t)$ be a shadow system of $K_i$, as in \eqref{def_shadowsystem} each in the same direction and defined in a common interval $[-a,b]$. Set $$\mathscr{K}(t)=(K_0(t),K_1(t),\dots,K_m(t)).$$ Then,
    $t\mapsto \vol[nm]{D^{m,\circ}(\mathscr{K}(t))}^{-1}$ is convex.
\end{theorem}
Clearly, Theorem~\ref{t:convex_t} with the choice of $K_i=K$ for all $i$ together with Proposition~\ref{p:shep} yields another proof of Theorem~\ref{t:res_schneider_polar} in the origin-symmetric case. We will need the following corollary of the Borell-Brascamp-Lieb inequality (cf. \cite{CG06}). For $p\in\R$, we say a function $f:\R^n\mapsto\R_+$ is $p$-concave if, for every $x,y$ such that $f(x)f(y)>0$, it holds
\[f((1-\lambda)x+\lambda y) \geq \left[(1-\lambda)f(x)^p+\lambda f(y)^p\right]^\frac{1}{p}.\]
When $p \to 0$, one recovers log-concavity.
\begin{proposition}
\label{p:BBL}
    Let $F(x,y)$ be a non-negative, $p$-concave function on $\R^d\times \R^s$, $p\geq -1/d$. If the function given by
    \[y\mapsto \int_{\R^d}F(x,y)dx\]
    is well defined for every $y\in\R^s$, then it is a $\frac{p}{1+dp}$-concave function.
\end{proposition}

\noindent Next, we need the following change of variables formula; it was established by Campi and Gronchi \cite{CG06}.
\begin{lemma}
\label{l:change}
    Let $h$ be an even, integrable, positively $1$-homogeneous function on $\mathbb{S}^{d-1}$. Then, by fixing a unit vector $w$, we have
    \begin{equation}
        \int_{\mathbb{S}^{d-1}}h^{-d}(\theta)d\theta = 2\int_{w^\perp}h^{-d}(x+w)dx.
        \label{eq:change_5}
    \end{equation}
\end{lemma}

\begin{proof}[Proof of Theorem~\ref{t:convex_t}]
    We start by observing that, from the origin-symmetry of $K_0$, we obtain from \eqref{eq:DM_sup}
    \begin{equation}
        h_{D^m(\mathscr{K})}((\theta_1,\cdots,\theta_m)) = h_{K_0}\left(\sum_{i=1}^m\theta_i\right) + \sum_{i=1}^mh_{K_i}(\theta_i).
        \label{eq:DM_sup_2}
    \end{equation}

Fix a vector $v \in \s$.
From Lemma~\ref{l:change}, with $d=nm$, applied to the function $h_{D^m(\mathscr{K})}$ and $w= \frac{1}{\sqrt{m}}(v,\dots,v)$, we have
\begin{equation}
\label{w^perp_int}
\begin{split}
    &\Vol{D^{m,\circ}(\mathscr{K})}
    =\frac{1}{nm}\int_{\mathbb{S}^{nm-1}}h_{D^m(\mathscr{K})}^{-nm}(\theta)d\theta \\
    &=\frac{2}{nm}\int_{{w}^{\perp}}h_{D^m(\mathscr{K})}^{-nm}(x+ w)dx \\
    &=\frac{2}{nm} \int_{w^\perp}\!\left[ \!h_{K_0}\!\left(\!\sqrt{m}v\!+\!\sum_{i=1}^m x_i\!\right) \!+\! \sum_{i=1}^mh_{K_i}\!\left(x_i\!+\!\frac{v}{\sqrt{m}}\!\right)\!\right]^{-nm}\!\!\!\!\!\!\!dx.
\end{split}
\end{equation}

For each $i$, let $K_i(t)$ be a shadow system of $K$ in the direction $v$. Then, replace each $K_i$ with $K_i(t)$ in \eqref{w^perp_int} and use the relation \eqref{eq:lifted_relation} to obtain
\begin{equation}
\begin{split}
    &\frac{mn}{2}\Vol{D^{m,\circ}(\mathscr{K}(t))}
    \\
    &=\int_{w^\perp}\!\!\!\left[h_{\tilde K_0}\left(\sqrt{m}(v \!+\!te_{n+1})\!+\!\sum_{i=1}^m x_i\!\right) \!+\!  \sum_{i=1}^mh_{\tilde K_i}\!\left(\!x_i\!+\!\frac{1}{\sqrt{m}}(v \!+te_{n+1})\right)\!\right]^{-nm}\!\!\!\!\!\!\!\!\!dx.
    \end{split}
\end{equation}
Applying Proposition~\ref{p:BBL}, with $p=-\frac{1}{nm}$ and $d=nm-1$, we see that the function $t\mapsto \Vol{D^{m,\circ}(\mathscr{K}(t))}$ is $\frac{-\frac{1}{nm}}{1-\frac{nm-1}{nm}}=-1$ concave.
\end{proof}

\subsection{Proof of the symmetric case using fiber symmetrization}
\label{sec:fiber}
In this section, we apply the operators $D^m$ and $D^{m,\circ}$ to only one convex body $K$. We saw in Section~\ref{sec:og_shadows} that the volume of $D^{m,\circ}(K)$ is monotone decreasing with respect to Steiner symmetrization on $K$.
It is natural then to understand the appropriate analogue of Steiner symmetrization in the space $\R^{nm}$, and the behavior of $D^{m,\circ}(K)$ with respect to it.
We show a third proof of Theorem \ref{t:res_schneider_polar} because a by-product of this proof is an important property of fiber-symmetrization of independent interest (Corollaries \ref{res_polarvolume} and \ref{res_preservedvolume}).

For a vector $x=(x_1,\dots,x_m)\in\R^{nm}$ and $v\in\R^n$, we write $$x^t v:=(\langle x_1,v \rangle,\dots,\langle x_m,v \rangle)\in\R^m.$$
The notation comes from regarding $x$ as an $n \times m$ matrix with columns $x_1, \ldots, x_m$. For a fixed vector $v \in \R^n\setminus\{o\}$ there is an orthogonal direct sum decomposition $\R^{nm} = v^m \times v^{\perp m}$ with
\begin{equation}
\label{eq:decomposition}
\begin{split}
	v^m &= \{v s^t: s \in \R^m\} \subseteq \R^{nm},
    \\
    &\text{and}\\
	v^{\perp m} &= \{x \in \R^{nm}: x^tv = o\in\R^m\},
\end{split}
 \end{equation}
where $$vs^t=(s_1 v, \ldots, s_m v)$$ is the exterior product of $v\in\R^n$ and $s\in\R^m$.

Recently, the fiber symmetrization with respect to this decomposition was used to prove analogues of the Petty projection inequality, \cite{HLPRY25,HLPRY23_2}, in Schneider's setting. 
There are actually two such fiber symmetrizations that can be defined with respect to the decomposition \eqref{eq:decomposition}.
\begin{definition}
Fix $n,m\in\N$. Let $v \in \R^n$ and let $L \subseteq \R^{nm}$ be a convex body. Then, we have the usual $m$th-order fiber symmetrization,
	\label{def_Ss}
\begin{equation}
\label{eq_sym_def}
\S L \!=\! \left\{x \!+\! v \left(\frac {s_1-s_2}{2}\right)^t : \!x^tv \!=\! o, s_i\! \in \R^m, x\!+\!v s_i^t \in L \text{ for } i=1,2\right\},\end{equation}
and the adjoint version,
\[\So L = \left\{\frac {x_1-x_2}{2} + v s^t : x_i^t v = o, s \in \R^m, x_i+v s^t \in L \text{ for } i=1,2\right\}.\]
\end{definition}

The first named author in \cite{JH23} proved that Euclidean balls minimize the \textit{mean width} of $D^m(K)$. The key was the following result.
\begin{lemma}[{\cite[Theorem 4.6]{JH23}}]
	\label{l:schneider_difference_body}
	If $K$ is a convex body in $\R^n$ and $v\in\s$, then \[D^m(S_v K) \subseteq \S(D^m (K)).\]
\end{lemma}

The reason why Lemma \ref{l:schneider_difference_body} implies that $w(D^m(K))$ is minimized by Euclidean balls, is that, for a convex body $L\subset\R^{nm}$, $w(\S L) \leq w(L)$ (see \cite[Proposition 2.4]{JH23}). In this section we show that the volume of the polar of a convex set in $\R^{nm}$ satisfies the same monotonicity property, and this fact implies Theorem \ref{t:res_schneider_polar}. First we show:
\begin{proposition}
	\label{res_polar}
	Fix $n,m\in\N$. If $L\subset\R^{nm}$ is an origin-symmetric convex body and $v\in\s$, then $\So (L^\circ) \subseteq (\S L)^\circ$.
\end{proposition}
\begin{proof}
Let $x+v q^t, y+vq^t \in L^\circ$ and $z + v s^t, z + v r^t \in L$.
	Since also $-y-vq^t \in L^\circ$,
	\begin{align}
		&\left\langle \frac {x-y}{2} + v q^t, z + v\left(\frac {s-r}{2}\right)^t \right\rangle 
		= \frac 12 \left( \langle x-y, z \rangle + \langle vq^t, v(s-r)^t \rangle \right) \\
		&= \frac 12 \left( \langle x, z \rangle - \langle y, x \rangle + \langle vq^t, v s^t \rangle - \langle v q^t, v r^t \rangle \right) \\
		&= \frac 12 \left( \langle x + vq^t, z + v s^t \rangle + \langle - y - v q^t, z + vr^t \rangle \right) \leq 1,
	\end{align}
	which, for every $z + v s^t, z + v r^t \in L$, means (see Definition \ref{def_Ss}) that $\frac 12(x-y) + vq^t \in (\S L)^\circ$.
	Since this is true for every $x+v q^t, y+v q^t \in L^\circ$, we get the result.
\end{proof}

Unlike the usual Steiner symmetrization (which is the operator $\S$ in the case $m=1$), fiber symmetrization does not always preserve volume. For a convex body $L\subset \R^{nm}$ and $v \in \s$, one has (see \cite[Lemma 3.1]{UJ23})
\begin{equation}
\label{res_volume_increasing}
	    \vol[nm]{\bar S_v L} \geq \vol[nm]{L} \quad \text{and} \quad 
	\Vol{\So L} \geq \Vol{L}.
\end{equation}
These inequalities are actually just a direct application of Fubini's theorem and the Brunn-Minkowski inequality; for completeness, we show the first inequality.
By the Brunn-Minkowski inequality, for every $x\in v^{\perp m}$,
\[
\vol[m]{\frac{1}{2} D(L \cap (x+ v^m ))} \geq \vol[m]{L \cap (x+ v^m)}.
\]
Applying Fubini's theorem, we obtain
\begin{align*}
    \vol[nm]{\bar S_{v} L}
    &= \int_{v^{\perp m}} \vol[m]{x+ \frac{1}{2} D(L \cap ( x+v^m))} dx \\
    &= \int_{v^{\perp m}} \vol[m]{\frac{1}{2} D(L \cap (x+v^m))} dx \\
    &\geq \int_{v^{\perp m}} \vol[m]{L \cap (x+v^m)} dx = \vol[nm]{L},
\end{align*}
as required.

\begin{corollary}
	\label{res_polarvolume}
	Fix $n,m\in\N$. If $L\subset\R^{nm}$ is an origin-symmetric convex body and $v\in\s$, then 
    \[\Vol{(\S L)^\circ} \geq \Vol{L^\circ}.\]
\end{corollary}
\begin{proof}
	By inequality \eqref{res_volume_increasing} and Proposition \ref{res_polar},
	\[\Vol{L^\circ} \leq \Vol{\So L^\circ} \leq \Vol{(\S L)^\circ}.\]
\end{proof}

Combining all these facts, we can now provide another proof of the inequality in Theorem~\ref{t:res_schneider_polar} in the symmetric case:

\begin{proof}[Third proof of Theorem \ref{t:res_schneider_polar}]
By Lemma \ref{l:schneider_difference_body}, we have the set inclusion $$(\S D^m (K))^\circ \subseteq D^{m,\circ} (S_v K).$$
	Observe that $K$ being symmetric yields $D^m (K)$ is origin-symmetric. Then by Corollary \ref{res_polarvolume}  we obtain
    \[\Vol{D^{m,\circ}(K)} \leq \Vol{(\S D^m (K))^\circ} \leq \Vol{D^{m,\circ} (S_v K)}.\]

	Finally, iterate this procedure along a sequence of directions that transforms $K$ in a ball of the same volume.
\end{proof}
We end this section with an application to fiber-symmetrization. As mentioned, the operator $\bar S_v$ does not preserve the volume. We can balance this bad behavior with the following property.
\begin{corollary}
\label{res_preservedvolume}
    There exists a universal constant $c_0>1$ such that if $v_1, \ldots, v_k \in \s$ with $k \geq 1$ and $L \subset \R^{nm}$ is an origin-symmetric convex body, then
    \[\vol[nm]{\bar S_{v_k} \circ \cdots \circ \bar S_{v_1} L}^{1/(nm)} \leq c_0 \vol[nm]{L}^{1/(nm)}\]
\end{corollary}
\begin{proof}
    By Corollary \ref{res_polarvolume} applied $k$-times,
    \[\vol[nm]{(\bar S_{v_k} \circ \cdots \circ \bar S_{v_1} L)^\circ} \geq \vol[nm]{L^\circ}.\]
    Then, the claim follows by using the Bourgain-Milman inequality \eqref{eq:VM}, the Blaschke-Santal\'o inequality \eqref{eq:BS} and Stirling's approximation \eqref{eq:stirling}.
\end{proof}

\section{A functional extension of the
Polar Schneider inequality}
\label{sec:functions}
In this section, we establish Theorem~\ref{t:functional_singular_polar_sch}. Like when we passed from Theorem~\ref{t:res_schneider} to Theorem~\ref{t:res_schneider_polar}, we will establish a more general inequality for $(m+1)$ functions. We recall a function $g$ is a centered Gaussian if there exists a positive definite, symmetric, $n\times n$ matrix $A$ with real entries such that $g(x)=e^{-\frac{1}{2}\langle Ax,x \rangle}$. We say $A$ is the generating matrix of $g$.
\begin{theorem}
\label{t:functional_polar_sch}
     Let, for $i=0,1,\dots,m,$ $f_i:\R^n\to\R_+$ be an integrable function such that either $f_i$ or $f_i^\circ$ has center of mass at the origin. Then,
     \begin{equation}
     \label{eq:PSFBS_0}
     \begin{split}
     \prod_{i=0}^{m}\left(\int_{\R^n}f_i^\frac{m+1}{m}\right)^\frac{m}{m+1} &\cdot \int_{\R^{nm}}\left(\prod_{i=1}^m f_i^\circ(x_i)\right)\left(f_{0}^\circ\left(-\sum_{i=1}^m x_i\right)\right)dx
     \\
     &\leq \left(\frac{2\pi m}{m+1}\right)^\frac{mn}{2}\left(\frac{2\pi}{(m+1)^\frac{1}{m}}\right)^\frac{nm}{2},
     \end{split}
     \end{equation}
     with equality if and only if there exists a centered Gaussian $g$ and constants $C_0,\dots, C_{m}>0$ such that $f_i=C_ig$ for all $i=0,\dots,m$. By the H\"older's  inequality, we also have, with the same assumptions and equality conditions,
\begin{equation}\begin{split}
     \left(\int_{\R^n}\prod_{i=0}^mf_i^\frac{1}{m}\right)^m &\cdot \int_{\R^{nm}}\left(\prod_{i=1}^m f_i^\circ(x_i)\right)\left(f_{0}^\circ\left(-\sum_{i=1}^m x_i\right)\right)dx
     \\
     &\leq \left(\frac{2\pi m}{m+1}\right)^\frac{mn}{2}\left(\frac{2\pi}{(m+1)^\frac{1}{m}}\right)^\frac{nm}{2}.
     \end{split}
\label{eq:sharp_2}
\end{equation}
\end{theorem}

We list as a corollary the case $m=1$.
\begin{corollary}
     \label{c:m_1}
         For $i=0,1$, let $f_i:\R^n\to\R_+$ be an integrable function such that either $f_i$ or $f_i^\circ$ has center of mass at the origin. Then,
\begin{equation}\left(\int_{\R^n}f_0^2\right)^\frac{1}{2}\left(\int_{\R^n}f_1^2\right)^\frac{1}{2}\left(\int_{\R^n}f_1^\circ(x)f_0^\circ(-x)dx\right) \leq \pi^n,
\label{eq:m_1_sharp}
\end{equation}
with equality if and only if there exists a centered Gaussian $g$ and $C_0,C_1>0$ such that $f_0=C_0g$ and $f_1=C_1g$. By the Cauchy-Schwarz inequality, we also have, with the same assumptions and equality conditions,
\begin{equation}\left(\int_{\R^n}f_0f_1\right)\left(\int_{\R^n}f_1^\circ(x)f_0^\circ(-x)dx\right) \leq \pi^n.
\label{eq:m_1_sharp_2}
\end{equation}
     \end{corollary}
\subsection{The functional Polar Schneider inequality} 
We prove Theorem~\ref{t:functional_polar_sch} in two steps. First, we show the existence of a finite maximizing constant, and ensure that the constant is obtained on the set of centered Gaussians. Then, we compute the constant and determine that each maximizing centered Gaussian must be the same.

\begin{lemma}
\label{l:functional_polar_sch}
    Let, for $i=0,\dots,m,$ $f_i:\R^n\to\R_+$ be an integrable function such that either $f_i$ or $f_i^\circ$ has center of mass at the origin. Then, there exists a sharp constant $C(n,m)\in (0,\infty),$ depending only on $n$ and $m$, such that
     \begin{equation}
     \label{eq:PSFBS}
     \prod_{i=0}^{m}\!\left(\!\int_{\R^n}f_i^\frac{m+1}{m}\!\right)^\frac{m}{m+1} \cdot \int_{\R^{nm}}\!\left(\prod_{i=1}^m f_i^\circ(x_i)\right)\!\left(f_{0}^\circ\left(-\sum_{i=1}^m x_i\right)\!\right)dx \leq C(n,m),
     \end{equation}
     and the equality is obtained uniquely on the set of centered Gaussians.
\end{lemma}

To prove Lemma~\ref{l:functional_polar_sch}, we will need two major theorems. The first one is the following extension of the functional Blaschke-Santal\'o inequality by Fradelizi and Meyer \cite{FM07} in the log-concave case and then Lehec \cite{JL08} in general; the case \eqref{eq:fun_sant} is recovered by setting $g=f^\circ$ and $\rho(t)=e^{-\frac{t}{2}}$.
\begin{proposition}
\label{prop:Ball_Fradelizi_Meyer}
Let $f,g:\R^n\to\R_+$ be integrable functions such that either $f$ or $g$ has center of mass at the origin. Suppose there exists a measurable function $\rho:\R\to\R_+$ such that the following inequality holds
$$f(x)g(y)\le \rho^{2}(\langle x,y\rangle ) \hbox{ for every  $x,y\in
        \R^n$ with $\langle x,y\rangle >0$}.$$
Then,
         \begin{equation}
         \label{eq:Ball_Fradelizi_Meyer}
         \int_{\R^n} f(x)dx\int_{\R^n} g(y)dy\le
        \left( \int_{\R^n} \rho({|x|^{2}})dx\right)^2.
        \end{equation}
\end{proposition}

We also need the Barthe-Brascamp-Lieb inequality. Fix $d,k\in\N\cup\{0\}$ and $d_0,\dots,d_k\in \N\cup\{0\},$ with $d_i\leq d$. We then define the \textit{Brascamp-Lieb datum} as $\mathscr{B}=(B_0,\dots,B_k)$ and $\mathscr{C}=(c_0,\dots,c_k)$, where each $B_i:\R^{d}\to\R^{d_i}$ is a surjective linear map and the constants $c_0,\dots,c_k >0$ are so that $\sum_{i=0}^kc_id_i=d$.  Define the Brascamp-Lieb functional on $\mathscr{H}=(h_0,\dots,h_k)\in (L^1(\R^n))^{k+1},$ with $h_i$ non-identically zero and nonnegative for all $i$, as
\begin{equation}
\operatorname{BL}(\mathscr{B}, \mathscr{C} ; \mathscr{H}):=\frac{\int_{\mathbb{R}^{d}}  \prod_{i=0}^k h_i\left(B_i x\right)^{c_i} d x}{\prod_{i=0}^k\left(\int_{\mathbb{R}^{d_i}} h_i d x_i\right)^{c_i}}, \in(0, \infty].
\label{eq:BLB}
\end{equation}
Then, if $\cap_{i\leq k} \text{ker} B_i = \{o\}$, the supremum of $\operatorname{BL}(\mathscr{B}, \mathscr{C} ; \mathscr{H})$ over all such collections of functions $\mathscr{H}$ is finite and equals the supremum over all such collections of the form $(g_0,\dots,g_k),$ where each $g_i$ is a centered Gaussian. This was shown by Brascamp and Lieb \cite{BL76:2} when each $d_i=1$ and by Barthe \cite{BF98} in the general case. 

\begin{proof}[Proof of Lemma~\ref{l:functional_polar_sch}]
Throughout the proof, we suppress the dependence of the maximizing constant on the datum $\mathscr{B}$ and $\mathscr{C}$. Set in \eqref{eq:BLB} $d=nm,$ $k=m,$ and $d_i=n$ for all $i$ to obtain
    that
    $$\frac{\int_{\R^{nm}}\left(\prod_{i=1}^m h_i(B_ix)^{c_i}\right)\left(h_{0}(B_{0}x)^{c_{0}}\right)dx}{\prod_{i=0}^{m}\left(\int_{\R^n}h_i\right)^{c_i}}\leq C_{BL},$$
    where $C_{BL}$ is the constant obtained by taking supremum over Gaussians.
    Next, define the symbols $r_i:=h_i^{c_i}$ to obtain
    $$\frac{\int_{\R^{nm}}\left(\prod_{i=1}^m r_i(B_ix)\right)\left(r_{0}(B_{0}x)\right)dx}{\prod_{i=0}^{m}\left(\int_{\R^n}r_i^\frac{1}{c_i}\right)^{c_i}}\leq C_{BL}.$$
    We then set, for $i=1,\dots,m,$ $B_i x = x_i$ where $x=(x_1,\dots,x_m)$ and $B_{0}x=-\sum_{i=1}^m x_i$, and for all $i$, set the exponents $c_i = \frac{m}{m+1}$ to obtain
    $$\frac{\int_{\R^{nm}}\left(\prod_{i=1}^m r_i(x_i)\right)\left(r_{0}\left(-\sum_{i=1}^m x_i\right)\right)dx}{\prod_{i=0}^{m}\left(\int_{\R^n}r_i^\frac{m+1}{m}\right)^{\frac{m}{m+1}}}\leq C_{BL}.$$
   Observe that $\sum_{i=0}^k c_id_i = \sum_{i=0}^{m}\left(\frac{m}{m+1}\right)n=nm=d,$ and thus $C_{BL}$ is finite. We then set $r_i=f_i^\circ$ for all $i$ to obtain
   $$\int_{\R^{nm}}\left(\prod_{i=1}^m f_i^\circ(x_i)\right)\left(f_0^\circ\left(-\sum_{i=1}^m x_i\right)\right)dx\leq C_{BL} \prod_{i=0}^{m}\left(\int_{\R^n}(f_i^\circ)^\frac{m+1}{m}\right)^\frac{m}{m+1}.$$
   Next, multiply both sides by $\prod_{i=0}^{m}\left(\int_{\R^n}f_i^\frac{m+1}{m}\right)^\frac{m}{m+1}$ to obtain
   \begin{align*}\prod_{i=0}^{m}\left(\int_{\R^n}f_i^\frac{m+1}{m}\right)^\frac{m}{m+1}&\left(\int_{\R^{nm}}\left(\prod_{i=1}^m f_i^\circ(x_i)\right)\left(f_{0}^\circ\left(-\sum_{i=1}^m x_i\right)\right)dx\right)
   \\
   &\leq C_{BL} \prod_{i=0}^{m}\left(\left(\int_{\R^n}(f_i^\circ)^\frac{m+1}{m}\right)\left(\int_{\R^n}f_i^\frac{m+1}{m}\right)\right)^\frac{m}{m+1}.
   \end{align*}
   We conclude by applying $(m+1)$ times the Proposition~\ref{prop:Ball_Fradelizi_Meyer} to the functions $f_i^\circ$ and $f_i$ with $\rho(t)=e^{-\frac{m+1}{m}\frac{t}{2}}$.

As for the equality cases, since each $f_i$ satisfies $f_i\leq f_i^{\circ\circ}$ and $f_i^{\circ\circ\circ}=f_i^{\circ}$, and thus the left-hand side increases when $f_i$ is replaced by $f_i^{\circ\circ},$ we must have that each $f_i$ is log-concave (since $f^\circ$ is always log-concave). Then,  the equality conditions for Proposition~\ref{prop:Ball_Fradelizi_Meyer} from \cite[Theorem 8]{FM07} and the equality conditions of the Barthe-Brascamp-Lieb inequality complete the proof.
\end{proof}

We now move on to the second part of proving Theorem~\ref{t:functional_polar_sch}, which is computing $C(n,m)$. We start by writing $C(n,m)$ as an optimization problem;  denote by $I_n$ the $n\times n$ identity matrix.
\begin{proposition}
\label{p:opt}
    Fix $n,m\in\N$ and let $C(n,m)$ be the constant from Theorem~\ref{t:functional_polar_sch}. Then, 
    $$C(n,m) = \left(\frac{4\pi^2 m}{m+1}\right)^\frac{nm}{2} \sup \mathcal{F}_{n}(A_0,\cdots,A_m),$$
    where the supremum is taken over $(m+1)$ tuples of positive-definite symmetric matrices and $\mathcal{F}_{n}$ is given by
    \begin{equation}
    \label{eq:F_formula}
            \mathcal{F}_n(A_0,\cdots,A_{m}) =\prod_{i=0}^m\operatorname{det}(A_i)^{\frac{m}{2(m+1)}}\operatorname{det}\left(M\right)^{-\frac{1}{2}}.
        \end{equation}
    with $M=M(A_0,\dots,A_m)$ defined in \eqref{eq:B_matrix} below.
\end{proposition}
\begin{proof}
We first recall the following well-known formula: for a symmetric, positive definite $n\times n$ matrix $A$, one has
     \begin{equation}
\label{eq:Gaussian_computation}
   \int_{\mathbb{R}^n} e^{-\frac{1}{2}\langle A x, x\rangle} d x=\frac{(2\pi)^\frac{n}{2}}{\operatorname{det}(A)^\frac{1}{2}}.
   \end{equation}
Applying formula \eqref{eq:Gaussian_computation} to $\frac{m+1}{m}A^{-1}$ instead of $A$, we see that
        \begin{equation}
        \label{eq:m_gaussian_computation}
        \begin{split}
        \int_{\R^n}e^{-\frac{1}{2}\frac{m+1}{m}\langle A^{-1}x,x\rangle dx}
        &= \left(\frac{2\pi m}{m+1}\right)^\frac{n}{2}\operatorname{det}(A)^{\frac{1}{2}}.
        \end{split}
        \end{equation}
        It is well-known and easily verifiable that $\left(e^{-\frac{1}{2}\langle A^{-1}\cdot,\cdot \rangle}\right)^\circ(x) = e^{-\frac{1}{2}\langle Ax,x \rangle}$. Also, we have for $c>0$ that $(cf)^\circ = \frac{1}{c}f^\circ$. Then, by setting each $f_i$ in Theorem~\ref{t:functional_polar_sch} to be a centered Gaussian generated by some $A_i^{-1}$ and using \eqref{eq:m_gaussian_computation}, we must maximize
        \begin{equation}
        \begin{split}
        \label{eq:almost_max_function}
            &\mathcal{G}_n(A_0,\cdots,A_{m}) := \left(\frac{2\pi m}{m+1}\right)^\frac{nm}{2}\prod_{i=0}^m\operatorname{det}(A_i)^{\frac{m}{2(m+1)}}
            \\
            &\times\int_{(\R^n)^m}e^{-\frac{1}{2}\left(\sum_{i=1}^m\langle A_ix_i,x_i \rangle+\langle A_0\left(\sum_{i=1}^mx_i\right),\left(\sum_{i=1}^mx_i\right) \rangle\right)}dx_1\dots dx_m.
        \end{split}
        \end{equation}
        
       Define the block matrix $M$ by
        \begin{equation}
        \label{eq:B_matrix}
		M \!=\! \left(
	\begin{array}{cccc}	A_1 + A_0 & A_0 &  \cdots & A_0 \\
A_0 & A_2 + A_0 &   \cdots & A_0 \\
\vdots & \vdots & \ddots & \vdots \\
A_0 & A_0 &  \cdots &A_m + A_0
	\end{array}
	\right).
	\end{equation}
    Then, \eqref{eq:almost_max_function} becomes
    \begin{equation}
        \begin{split}
        \label{eq:almost_max_function_2}
            \mathcal{G}_n(A_0,\!\cdots\!,A_{m}) = &\left(\frac{2\pi m}{m+1}\right)^\frac{nm}{2}\prod_{i=0}^m\operatorname{det}(A_i)^{\frac{m}{2(m+1)}}\!\!\!\int_{\R^{nm}}\!\!\!\!\!\!\!\!e^{-\frac{1}{2}\left\langle Mz,z \right\rangle}dz.
        \end{split}
        \end{equation}
        We obtain from \eqref{eq:almost_max_function_2} and \eqref{eq:Gaussian_computation} the claimed formula by setting $$\mathcal{F}_n(A_0,\cdots,A_m) := \left(\frac{m+1}{4\pi^2 m}\right)^\frac{nm}{2}\mathcal{G}_n(A_0,\!\cdots\!,A_{m}).$$
\end{proof}

We can give an analytic description of the matrix $M$ from Proposition~\ref{p:opt}. Recall that $\Delta_m$ denotes the diagonal embedding. Let $C_1, \ldots, C_n:\R^n\to\R^{nm} $ be the coordinate inclusions, i.e., $C_i(x)$ has $x$ in the $i$th block and $o$ in the other blocks. We can think of $\Delta_m$ and the $C_i$ as $nm \times n$ matrices, so that $C_i^T$ is the projection to the $i$th block and $\Delta_m^T(x) = \sum_{i = 1}^m x_i$. Then, the matrix $M$ satisfies
\[M=\Delta_m A_0 \Delta_m^T + \sum_{i = 1}^m C_i A_i C_i^T.\]
In the next proposition, we compute the determinant of $M$.
\begin{proposition}
\label{p:M_cal}
    Fix $m,n\in\N$. For positive-definite $n\times n$ matrices $A_0,\dots,A_m$, let $M$ be the block matrix given by \eqref{eq:B_matrix}. Then, 
$$\det(M) = \prod_{i = 0}^m \det(A_i) \cdot  \det\left( \sum_{i = 0}^m A_i^{-1}\right).$$
\end{proposition}
\begin{proof}Let $D$ be the block diagonal matrix with $A_1, \ldots, A_m$ on the diagonal. Then, $M = D^{1/2} N D^{1/2}$ where $N = I_{nm} + SS^T$ and $S = D^{-1/2} \Delta_m A_0^{1/2}$, i.e. $N$ is given by
\begin{align*}
    \begin{pmatrix} 
I_n + A_1^{-1/2} A_0 A_1^{-1/2} & A_2^{-1/2} A_0 A_1^{-1/2} & \cdots & A_m^{-1/2} A_0 A_1^{-1/2} \\
A_1^{-1/2} A_0 A_2^{-1/2} & I_n + A_2^{-1/2} A_0 A_2^{-1/2} & \cdots & A_m^{-1/2} A_0 A_2^{-1/2} \\
\vdots & \vdots & \ddots & \vdots \\
A_1^{-1/2} A_0 A_m^{-1/2} & A_2^{-1/2} A_0 A_m^{-1/2} & \cdots & I_n + A_m^{-1/2} A_0 A_m^{-1/2} 
\end{pmatrix}.
\end{align*}
So, using the identity 
$$\det(I_k + LL^T) = \det(I_\ell + L^TL)$$ 
for $L$ a $k \times \ell$ matrix, we obtain 
$$\det(N) = \det(I_n + S^T S) = \det(I_n + A_0^{1/2} \Delta_m^T D^{-1} \Delta_m A_0^{1/2}).$$
Now, $\Delta_m^T D^{-1} \Delta_m$ is just $\sum_{i = 1}^m A_i^{-1}$, so we may write $$I_n + S^T S = A_0^{1/2} \left(\sum_{i = 0}^m A_i^{-1}\right) A_0^{1/2}.$$ Hence, we obtain $\det(N) = \det(A_0) \cdot \det\left(\sum_{i = 0}^m A_i^{-1}\right)$, yielding
$$\det(M) = \det(D) \det(N) = \prod_{i = 0}^m \det(A_i) \cdot \det\left(\sum_{i = 0}^m A_i^{-1}\right),$$
as claimed.
\end{proof}

Using Proposition~\ref{p:M_cal}, we solve the optimization problem introduced in Proposition~\ref{p:opt}.
\begin{corollary}
\label{c:constant}
    Fix $m,n\in\N$. Let $\mathcal{F}_n$ be the function given by \eqref{eq:F_formula}. Then, $\sup \mathcal{F}_{n}(A_0,\cdots,A_m)$ is obtained if and only if each $A_i$ is the same $A$, for any $n\times n$ positive-definite symmetric matrix $A$, i.e. 
    \[
    \sup \mathcal{F}_{n}(A_0,\cdots,A_m) = \mathcal{F}_{n}(A,\cdots,A)= (m+1)^{-\frac{n}{2}}.
    \]
\end{corollary}
\begin{proof}
    From Propositions~\ref{p:opt} and \ref{p:M_cal}, we must maximize
    \begin{align*}
        \mathcal{F}_n(A_0,\dots,A_m)=\prod_{i = 0}^m \det(A_i)^{-\frac{1}{2(m+1)}} \cdot  \det\left( \sum_{i = 0}^m A_i^{-1}\right)^{-\frac{1}{2}}.
    \end{align*}
A special case of the Brunn-Minkowski inequality for determinants is 
\begin{equation}
\label{eq:BM_determin}
\operatorname{det}\left(\frac{1}{m+1}\sum_{i=0}^mA_i^{-1}\right) \geq \prod_{i=0}^m \operatorname{det}(A_i^{-1})^\frac{1}{m+1}= \prod_{i=0}^m\operatorname{det}(A_i)^{-\frac{1}{m+1}},\end{equation}
with equality if and only if $A_0=\cdots=A_m$ (see \cite[Lemma 2.1.5]{AGA}). By the $n$-homogeneity of the determinant, this implies the claim.    
\end{proof}

\begin{proof}[Proof of Theorem~\ref{t:functional_polar_sch}]
    From Lemma \ref{l:functional_polar_sch}, the inequality is established, and we know the maximizing value $C(n,m)$ is obtained only among centered Gaussians. From Proposition~\ref{p:opt} and Corollary~\ref{c:constant}, we know that
    $$C(n,m) = \left(\frac{2\pi m}{m+1}\right)^\frac{mn}{2}\left(\frac{2\pi}{(m+1)^\frac{1}{m}}\right)^\frac{mn}{2},$$
    and that equality is obtained if and only if each function is a multiple of the same centered Gaussian.
\end{proof}

\begin{proof}[Proof of Theorem~\ref{t:functional_singular_polar_sch}]
    The claim is immediate from Theorem~\ref{t:functional_polar_sch} by setting each $f_i$ to be the same.
\end{proof}

\subsection{Comparing functional and geometric results}
\label{sec:compare}
In this section, we compare Theorems~\ref{t:res_schneider_polar} and \ref{t:functional_singular_polar_sch}, and, more generally, Theorems~\ref{t:res_schneider} and \ref{t:functional_polar_sch}. Recall that the dual $L^p$ sum, with $p\geq 1$, of two convex bodies $K$ and $L$ is the convex body $K {\tilde +_{-p}} L$ with gauge given by $\left(\|\cdot\|_K^p + \|\cdot\|_L^p\right)^\frac{1}{p}$. In particular, one has that $(DK)^\circ=K^\circ \tilde +_{-1} (-K)^\circ$. 
Indeed,
\begin{equation}
\label{eq:1-sum}
\|\cdot\|_{(DK)^\circ}= h_{DK} = h_K + h_{-K} = \|\cdot\|_{K^\circ} + \|\cdot\|_{(-K)^\circ}.
\end{equation}

Recalling the notion of $(-{K}_i)e^t_i$ from \eqref{eq:K_i}, we define the body
\[
D^{m,\circ}_2(\mathscr{K}):=\Delta_m (K_0^\circ)\;\tilde +_{-2}\left(\sum_{i=1}^m (-K_i^\circ)e^t_i\right),
\]
where the second summation means iterative $\tilde +_{-2}$ sums. In terms of Firey's $L^p$ extension of Minkowski summation \cite{Firey62}, one can verify that $D^{m}_2(\mathscr{K}):=(D^{m,\circ}_2(\mathscr{K}))^\circ$ is merely iterative $L^2$ Minkowski sums of $\Delta_m(K_0)$ and the $(-K_i)e^t_i$. Taking $f_i= e^{-\|\cdot\|_{-K_i}^2/2}$, where $K_i$ is a centered convex body, in Theorem~\ref{t:functional_polar_sch}, we get from \eqref{eq:measure_p_homo} (with $\mu$ the Lebesgue measure and $p=2$), \eqref{eq:polar_of_bodies} and \eqref{eq:ball_volume}:
 \begin{equation}
 \label{eq:implying_geo_theorem}
     \prod_{i=0}^m\vol[n]{K_i}^\frac{m}{m+1}\vol[nm]{D^{m,\circ}_2(\mathscr{K})}
     \leq \frac{\vol[n]{\B}^m\vol[nm]{B_2^{nm}}}{(m+1)^\frac{n}{2}},
 \end{equation}
with equality if and only if each $K_i$ is the same centered ellipsoid. We isolate the $m=1$ case:
\begin{equation}
\label{eq:implying_geo_theorem_m_1}
    \vol[n]{K_0}^\frac{1}{2}\vol[n]{K_1}^\frac{1}{2}\vol[n]{(D_2(K_0,K_1))^\circ} \leq 2^{-\frac{n}{2}}\vol[n]{\B}^2,
\end{equation}
with equality if and only if $K_0=K_1$ is a linear image of $\B$.
\begin{proposition}
    Fix $m,n\in\N$ and let, for $i=0,1,\dots,m,$ $K_i\subset\R^n$ be a convex body containing the origin in its interior and define the collection $\mathscr{K}=(K_0,\dots,K_m)$. Then, 
    \begin{equation}
\label{eq:dual_sum_inclsuion}
\left(m+1\right)^{-\frac{1}{2}}D^{m,\circ}_2(\mathscr{K})\subseteq D^{m,\circ}(\mathscr{K}).
\end{equation}
There is equality if and only if $m=1$ and $K_0=-K_1$.  In particular, \eqref{eq:implying_geo_theorem_m_1} implies the $m=1$ case of Theorem~\ref{t:res_schneider_polar} when $K_1=K_0$ is origin-symmetric.
\end{proposition}
\begin{proof}
Indeed, we have, by applying Jensen's inequality,
\begin{align*}
 &\|(x_1,\dots,x_m)\|_{D^{m,\circ}(\mathscr K)}^2 = \left(h_{K_0}\left(\sum_{i=1}^m x_i\right)+\sum_{i=1}^mh_{-K_i}(x_i)\right)^2 
 \\
 &= \left(\frac{1}{m+1}h_{(m+1)K_0}\left(\sum_{i=1}^m x_i\right)+\sum_{i=1}^m\frac{1}{m+1}h_{-(m+1)K_i}(x_i)\right)^2
 \\
 & \leq \frac{1}{m+1}h_{(m+1)K_0}\left(\sum_{i=1}^m x_i\right)^2 + \sum_{i=1}^m\frac{1}{m+1}h_{-(m+1)K_i}(x_i)^2
 \\
 & =  (m+1)\left(\left\|\sum_{i=1}^m x_i\right\|^2_{K_0^\circ}+\sum_{i=1}^m\|x_i\|_{(-K_i)^\circ}^2\right).
 \end{align*}
 The claim follows. The equality characterization comes from the fact that the use of Jensen's inequality will be strict when $m>1$, and, when $m=1$, it will have equality if and only if $h_{-K_1}=h_{K_0}$, which means $K_0=-K_1$. 
 \end{proof}
 
\subsection{An application to Poinecar\'e-type inequalities}
\label{sec:poin}
Let $\mu$ be a probability measure on $\R^n$. The variance of a nonnegative, measurable function $f$ on $\R^n$ with respect to $\mu$ is
\[
\operatorname{Var}_\mu f := \int_{\R^n}f(x)^2d\mu(x)-\left(\int_{\R^n}f(x)d\mu(x)\right)^2.
\]
We recall the standard Gaussian measure $\gamma_n$ on $\R^n$ is given by 
\[
d\gamma_n(x)=\frac{1}{(2\pi)^{\frac{n}{2}}}e^{-\frac{|x|^2}{2}}dx.
\]
The Poincar\'e inequality for the Gaussian measure asserts that, if $f$ is a $C^1$ smooth function that is integrable with respect to $\gamma_n$, then
\begin{equation}
\label{eq:classical_poincare}
    \operatorname{Var}_{\gamma_n} f \leq \int_{\R^n}|\nabla f(x)|^2d\gamma_n(x).
\end{equation}
It was shown by D. Cordero-Erausquin, M. Fradelizi and B. Maurey \cite{CEFM04} that if $f$ is also assumed to be even, then the inequality \eqref{eq:classical_poincare} improves by factor of $1/2$:
\begin{equation}
\label{eq:better_poincare}
    \operatorname{Var}_{\gamma_n} f \leq \frac{1}{2} \int_{\R^n}|\nabla f(x)|^2d\gamma_n(x).
\end{equation}
It is well-known (see e.g. \cite{CKLR24}) that one can linearize the classical functional Blashcke-Santal\'o inequality \eqref{eq:fun_sant} in the case of even functions to arrive at \eqref{eq:better_poincare}. Following this well-trodden path, we obtain a Poincar\'e-type inequality as an application of Theorem~\ref{t:functional_singular_polar_sch}. It will be convenient to write $f=e^{-\psi},$ in which case $f^\circ =e^{-\psi^\star}$, where $\psi^\star(x)=\sup_{z\in\R^n}(\langle x,z \rangle-\psi(z))$ is the Legendre transform of $\psi$. We introduce, for $\psi:\R^n\to\R$, the function $D^m\psi:\R^{nm}\to\R$ given by 
\begin{equation}D^m\psi(x_1,\dots,x_m) = \left(\sum_{i=1}^m \psi(x_i)+\psi\left(-\sum_{i=1}^m x_i\right)\right).
\label{eq:D_m_deff}
\end{equation}

We denote by $\gamma_n^\sigma$ the normal Gaussian measure on $\R^n$ with variance $\sigma>0$, that is 
\[
d\gamma_n^\sigma(x)=\frac{1}{(2\pi\sigma)^{\frac{n}{2}}}e^{-\frac{|x|^2}{2\sigma}}dx.
\]
Furthermore, we define the \textit{Schneider-Gaussian measure} $\gamma_{n,m}$ on $\R^{nm}$ by
\[
d\gamma_{n,m}(x)=\left(\frac{m+1}{\left(2\pi\right)^{m}}\right)^\frac{n}{2}e^{-D^m\frac{|\cdot|^2}{2}(x)}dx.
\]

\begin{corollary}
    Fix $m,n\in\N$. Let $\psi:\R^n\to \R^+$ be an even $C^1$ function with $\psi,|\nabla \psi|^2\in L^1(\R^n,\gamma_n^\frac{m}{m+1})$ and $D^m\psi,|\nabla D^m\psi|^2\in L^1(\R^{nm},\gamma_{n,m})$. Then,
    \begin{align*}
         &\left(\frac{m+1}{m}\right)\operatorname{Var}_{\gamma_n^\frac{m}{m+1}}\psi+\frac{1}{m+1}\operatorname{Var}_{\gamma_{n,m}} D^m \psi
         \\
         &\leq\frac{1}{2}\Bigg(\frac{1}{m+1}\int_{\R^{nm}}|\nabla D^m\psi(x)|^2 d\gamma_{n,m}(x)+\int_{\R^{n}}|  \nabla \psi(x)|^2 d\gamma_n^{\frac{m}{m+1}}(x) \Bigg).
     \end{align*}
\end{corollary}
Observe that $\nabla \psi$ refers to taking the gradient of $\psi$ on $\R^n$ while $\nabla D^m\psi$ contains the gradient on $\R^{nm}$.  The quantity $|\nabla D^m\psi|^2$ seems a bit mysterious, but in fact it is not so complicated. Since $\nabla$ respects the product structure of $\R^{nm}$, one has that $$\nabla D^m\psi(x_1,\dots,x_m)=\left\{\nabla \psi(x_i) - \nabla \psi\left(-\sum_{i=1}^m x_i\right)  \right\}_{i=1}^m.$$
     Thus,
     $$|\nabla D^m\psi(x_1,\dots,x_m)|^2=\sum_{i=1}^m\left|\nabla \psi(x_i) - \nabla \psi\left(-\sum_{i=1}^m x_i\right)\right|^2.$$
     
\begin{proof}
Using standard approximation techniques, and the fast decay of the densities of the measures $\gamma_n^{\frac{m}{m+1}}$ and $\gamma_{n,m}$, we may assume that $\psi$ is compactly supported. In Theorem~\ref{t:functional_singular_polar_sch}, write $f=e^{-V}$ and switch the roles of $V$ and $V^\star$. Then, one can take logarithm to obtain
     \begin{equation}
     \label{eq:PSFBS_3}
     \begin{split}
     &m\log\left(\int_{\R^n}e^{-\frac{m+1}{m}V^\star(x)}dx\right) 
     +\log\left(\int_{\R^{nm}}e^{-D^m(V(x))}dx\right) 
     \\
     &\leq m\log\left(\int_{\R^n}e^{-\frac{m+1}{m}\frac{|x|^2}{2}}dx\right) + \log \left(\int_{\R^{nm}}e^{-D^m\frac{|\cdot|^2}{2}(x)}dx\right).
     \end{split}
     \end{equation}
     
     We consider the case when $V(x)=V_\epsilon(x)=\frac{|x|^2}{2}+\epsilon \psi(x)$. Focusing on the second integral in \eqref{eq:PSFBS_3}, we have, by Taylor expanding $t\mapsto e^t$,
     \begin{align*}
         &\log \left(\int_{\R^{nm}}e^{-D^m(V_\epsilon(x))}dx\right) -\log \left(\int_{\R^{nm}}e^{-D^m\frac{|\cdot|^2}{2}(x)}dx\right)
         \\
         &= 
         \log\left(\int_{\R^{nm}}\text{exp}\left(-\epsilon D^m\left(\psi(x)\right)\right)d\gamma_{n,m}(x)\right)
         \\
         &=\log\left(1-\epsilon \int_{\R^{nm}} \left(D^m\psi(x) + \frac{\epsilon^2}{2}(D^m\psi(x))^2\right)d\gamma_{n,m}(x) + o(\epsilon^3)\right)
         \\
         &= - \epsilon\int_{\R^{nm}} D^m\psi(x)d\gamma_{n,m}(x)+\frac{\epsilon^2}{2} \left(\operatorname{Var}_{\gamma_{n,m}} D^m\psi\right)+o(\epsilon^3),
     \end{align*}
     
     where we Taylor expanded $t\mapsto \log(1-t)$ in the final line. We next need the well-known fact that if $\psi\in C^1(\R^n)$ is compactly supported, then $$V_\epsilon^\star = \frac{|x|^2}{2} - \epsilon \psi(x) + \frac{\epsilon^2}{2}|\nabla \psi|^2 + o(\epsilon^3),$$
     and the dependence of the term $o(\epsilon^3)$ on $x$ is uniform on $\text{supp}(\psi)$ (see e.g. \cite[Lemma 7.1]{CKLR24}). Then, for the first integral in \eqref{eq:PSFBS_3}, we have
     \begin{align*}
         &m\log\left(\int_{\R^n}e^{-\frac{m+1}{m}V_\epsilon^\star(x)}dx\right)-m\log\left(\int_{\R^n}e^{-\frac{m+1}{m}\frac{|x|^2}{2}}dx\right)
         \\
         & = m\log\left(\int_{\R^n}\text{exp}\left(\frac{m+1}{m}\left(\epsilon\psi(x)-\frac{\epsilon^2}{2}|\nabla \psi(x)|^2+o(\epsilon^3)\right)\right)d\gamma_n^\frac{m}{m+1}(x)\right)
         \\
         & = m\log\bigg(1+\int_{\R^n}\bigg(\left(\frac{m+1}{m}\right)\epsilon\psi(x)+\left(\frac{m+1}{m}\right)^2\frac{\epsilon^2}{2}\psi(x)^2
         \\
         &\quad-\left(\frac{m+1}{m}\right)\frac{\epsilon^2}{2}|\nabla \psi(x)|^2\bigg)d\gamma_{n}^\frac{m}{m+1}(x) +o(\epsilon^3)\bigg),
    \end{align*}
    where we Taylor expanded $t\mapsto e^t$ like before. Now, we Taylor expand $t\mapsto \log(1+t)$ to obtain
    \begin{align*}
    &m\log\left(\int_{\R^n}e^{-\frac{m+1}{m}V_\epsilon^\star(x)}dx\right)-m\log\left(\int_{\R^n}e^{-\frac{m+1}{m}\frac{|x|^2}{2}}dx\right)
         \\
         &=\epsilon(m+1)\int_{\R^{n}}\psi(x)d\gamma_n^{\frac{m}{m+1}}(x)
         \\
         &\quad+ \frac{\epsilon^2}{2}(m+1)\left(\left(\frac{m+1}{m}\right)\operatorname{Var}_{\gamma_n^\frac{m}{m+1}}\psi -\int_{\R^{n}}|\nabla \psi(x)|^2d\gamma_n^{\frac{m}{m+1}}(x)\right) + o(\epsilon^3). 
     \end{align*}
     Thus, the inequality \eqref{eq:PSFBS_3} reduces to
     \begin{align*}
         &\frac{\epsilon}{2}\left(\left(\frac{m+1}{m}\right)\operatorname{Var}_{\gamma_n^\frac{m}{m+1}}\psi -\int_{\R^{n}}|\nabla \psi(x)|^2d\gamma_n^{\frac{m}{m+1}}(x)\right)
         \\
         &\quad\quad\quad+\frac{\epsilon}{2(m+1)}\operatorname{Var}_{\gamma_{n,m}} D^m\psi+o(\epsilon^2) 
         \\
         &\leq \frac{1}{m+1}\int_{\R^{nm}} D^m\psi(x)d\gamma_{n,m}(x) - \int_{\R^{n}}\psi(x)d\gamma_n^{\frac{m}{m+1}}(x).
     \end{align*}

     In \eqref{eq:PSFBS_3}, we switch $V$ and $V^\star$ to obtain something similar to the above, but with $\psi$ and $D^m \psi$ switching roles, e.g. $\nabla D^m \psi$ appears instead of $\nabla \psi$. One gets
     \begin{align*}
         &\frac{\epsilon}{2(m+1)}\left(\operatorname{Var}_{\gamma_{n,m}} D^m\psi-\int_{\R^{nm}}| \nabla D^m\psi(x)|^2 d\gamma_{n,m}(x)\right) 
         \\
         &\quad\quad\quad+\frac{\epsilon}{2}\left(\frac{m+1}{m}\right)\operatorname{Var}_{\gamma_n^\frac{m}{m+1}}\psi +o(\epsilon^2) 
         \\
         &\leq -\left(\frac{1}{m+1}\int_{\R^{nm}} D^m\psi(x)d\gamma_{n,m}(x) - \int_{\R^{n}}\psi(x)d\gamma_n^{\frac{m}{m+1}}(x)\right).
     \end{align*}
     Thus, we can combine the two and obtain the claim when $\epsilon\to 0$.
     \end{proof}     

{\bf Acknowledgments:} Part of this work was done while all the authors were in residence at the Hausdorff Research Institute for Mathematics in Bonn, Germany, during Spring 2024 for the Dual Trimester Program: ``Synergies between modern probability, geometric analysis, and stochastic geometry'' and also while all the authors were at the Oberwolfach Research Institute for Mathematics in Germany for the workshop ``Convex Geometry and its Applications", ID 2451. Finally, we would thank Dario Cordero-Erausquin and Stanis\l aw Szarek, whose comments helped refine the notation and presentation of the facts herein.

{\bf Funding:} The first named author was supported by Grants RYC2021-031572-I and PID2022-136320NB-I00, funded by the Ministry of Science and Innovation and by the E.U. Next Generation EU/Recovery, Transformation and Resilience Plan. The second named author was supported by a fellowship from the FSMP Post-doctoral program and by the NSF MSPRF fellowship via NSF grant DMS-2502744. The third named author is supported by NSF-BSF DMS-2247834 and is also supported in part at the Technion by a fellowship from the Lady Davis Foundation. The fourth named author is supported by the ANR project DiAnQuGe and partially by a Chateaubriand Fellowship.

 \newpage
\bibliographystyle{acm}
\bibliography{references}

\end{document}